\documentclass{amsart}

\title[Gowers' Ramsey Theorem with multiple operations and the Lelek fan]{Gowers' Ramsey Theorem with multiple operations and dynamics of the homeomorphism group of the Lelek fan}

\author[D. Barto\v{s}ov\'{a}]{Dana Barto\v{s}ov\'{a}}
\address{Institute de Matem\'atica e Estat\'istica, Universidade de S\~ao Paulo, Brazil}
\email{dana@ime.usp.br}

\author[A. Kwiatkowska]{Aleksandra Kwiatkowska}
\address{Institut f\"{u}r Mathematische Logik, Universit\"{a}t  M\"{u}nster,  Germany and 
Instytut Matematyczny, Uniwersytet Wroc{\l}awski,   Poland}
\email{kwiatkoa@uni-muenster.de}



\usepackage{amsfonts,amssymb,amsmath,amsthm}
\usepackage{url}
\usepackage{xcolor}
\usepackage{enumerate}
\usepackage{comment}
\usepackage{graphics}
\usepackage{graphicx}
\usepackage{eqnarray}

\theoremstyle{plain}
\newtheorem{thm}{Theorem}[section]
\newtheorem{lemma}[thm]{Lemma}
\newtheorem{cor}[thm]{\bf Corollary}

\newtheorem{rem}[thm]{\bf Remark}

\theoremstyle{prop}
\newtheorem{prop}[thm]{\bf Proposition}

\newtheorem{question}[thm]{\bf Question}

\newtheorem*{claim*}{\bf Claim}
\newtheorem{definition}[thm]{\bf Definition}

\def\fin{{\rm{ FIN }}}
\def\supp{{\rm{ supp }}}

\def\tp{{\rm{ tp }}}
\def\id{{\rm{ id }}}
\def\min{{\rm{ min }}}
\def\he{{\rm{ ht }}}

\newcommand{\aut}{{\rm Aut}{(\lel)}}
\newcommand{\lel}{\mathbb{L}}
\newcommand{\auto}{{\rm Aut}{(\lel_<)}}

\def\g{{\mathbb{G}}}
\def\G{{\mathcal{ G}}}
\def\F{{\mathcal{ F}}}

\def\u{{\mathcal{ U}}}

\begin{document}

\begin{abstract}
We generalize the finite version of Gowers' Ramsey theorem to multiple tetris-like operations and apply it to show that 
{a group of homeomorphisms that preserve a  ``typical'' linear order of branches
of the Lelek fan, a  compact connected metric space with many symmetries,}  is extremely amenable.

\end{abstract}

\maketitle

\section{Introduction}\label{intro}

In \cite{G}, Gowers proved a generalisation of Hindman's finite sums theorem in order to show
the oscillation stability of the unit sphere in the Banach space $c_0.$ 
Recently, Tyros in \cite{T} and   Ojeda-Aristizabal in \cite{OA}  independently gave constructive combinatorial proofs of the finite version of  Gowers' theorem.

This article aims  to give a new Ramsey theorem, which generalizes the finite version of
Gowers' Ramsey theorem to multiple operations (Theorem \ref{Tstepup_d}), and most importantly
 to establish a surprising connection between our Ramsey theorem and the dynamics of the homeomorphism group of the Lelek fan -- a compact connected metric space with
many symmetries.

Our work was motivated by a striking correspondence between  structural Ramsey theory, Fra\"{i}ss\'{e} theory, and topological dynamics of automorphism groups, which was established by Kechris, Pestov and Todor\v{c}evi\'c in {\cite{KPT}},
 and further extended by Nguyen van Th\'e  in \cite{NVT}. In these articles, they characterized  a strong fixed point property, called extreme amenability, of automorphism groups
in terms of the Ramsey property. Here a topological group is \emph{extremely amenable} if every continuous action
on a compact Hausdorff space admits a fixed point.
 For instance, using the Ramsey property for linearly ordered finite metric spaces by Ne\v{s}et\v{r}il \cite{N}, 
the authors of \cite{KPT} showed that the isometry group of the separable Urysohn {metric} space  is extremely amenable. This result was originally
proved by Pestov in \cite{P} using  concentration of measure techniques. Further, 
applying the Ramsey property for finite linearly ordered graphs (Ne\v{s}et\v{ril}-R\"{o}dl, \cite{NR1} and \cite{NR2}), finite linearly ordered  hypergraphs (Ne\v{s}et\v{ril}-R\"{o}dl, \cite{NR1} and \cite{NR2}; Abramson-Harrington, \cite{AH}), and finite naturally ordered vector spaces over a finite field (Graham-Leeb-Rothschild, \cite{GLR}), 
Kechris, Pestov and Todor\v{c}evi\'c 
 showed that the groups of automorphisms of the random ordered graph, the random ordered hypergraph, and {the ordered}  $\aleph_0$-dimensional vector space over a finite field, respectively,
 are extremely amenable.

In this article, we dualize the Kechris-Pestov-Todor\v{c}evi\'c correspondence from \cite{KPT} to the projective Fra\"{i}ss\'{e} setting 
(Section \ref{KPTcorrespondence})
 and give its first application, namely  to the dynamics of a {certain natural} group of  homeomorphisms of the Lelek fan
(Section \ref{dynamics}).
 The projective Fra\"{i}ss\'{e} theory was originally developed by Irwin and Solecki in \cite{IS} in order to 
capture a well-known
 {compact and connected  metric space -- the pseudo-arc}.

\section{Discussion of results}

\subsection{Dynamics of the homeomorphism group of the Lelek fan}\label{dynamo}

A {\em continuum} is a compact connected  metric space. Denoting by $C$ the Cantor set and by $[0,1]$ the unit interval,
one defines the {\em Cantor fan} to be the 
quotient of $C\times [0,1]$ 
by the equivalence relation $\sim$ given by $(a,b)\sim (c,d)$ if and only if either $(a,b)=(c,d)$ or $b=d=0.$
For a continuum $X,$ a point $x\in X$ is an {\em endpoint} in $X$ if for every homeomorphic embedding $h:[0,1]\to X$ with $x$ in the image of $h$ either $x=h(0)$ or $x=h(1).$ The {\em Lelek fan} $L$, constructed by Lelek  in \cite{L}, can be characterized as the unique non-degenerate subcontinuum of the Cantor fan whose endpoints are dense (see \cite{BO} and \cite{C}).
Denote by $v$ the {\em top} $(0,0)/\!\!\sim$ of the Lelek fan. 
The ``endpoint'' and the ``top'' belong to the standard terminology in continuum theory. We point out that when we think of the Cantor fan,
the top point is often really at the bottom.

We will use the description of the Lelek fan via the class of finite fans as in \cite{BK}, where by a {\em fan} we mean an  undirected connected simple graph with all loops, with no cycles of the length greater than one,
and with a distinguished point $r$, called the {\em root}, such that all elements other than $r$ have degree at most 2. We will study the class 
$\F$ of finite fans and the class
$\F_<$ of finite fans expanded by a linear order on the set of branches. Families $\F$ and  $\F_<$ form projective Fra\"{i}ss\'e classes and therefore have  projective Fra\"{i}ss\'e limits, as defined by Irwin and Solecki \cite{IS} dualizing the classical (injective)  Fra\"{i}ss\'e theory from model theory. 
The projective Fra\"{i}ss\'e limit of $\F$ is the Lelek fan, as proved in  \cite{BK} (see also Section \ref{Lelekcons}), 
and the projective Fra\"{i}ss\'e limit of $\F_<$
will turn out to be a ``branch-ordered'' Lelek fan.
We provide necessary basics about projective Fra\"{i}ss\'e classes and projective Fra\"{i}ss\'e limits in 
Section \ref{Lelekcons}.

For a class $\G$ of finite structures and 
$A,B\in\G,$ we denote by ${B \choose A}$ the set of all epimorphisms (that is, surjective maps preserving the structure;
see Section  \ref{Lelekcons} for the definition of a structure and of an epimorphism) from $B$ onto $A.$
We say that $\G$  is a {\em Ramsey class} or that it has the {\em Ramsey property}  if for every $A,B\in \G$ and every natural number  $r\geq 2$ there exists $C\in \G$ such that for every colouring $c$ of ${C \choose A}$ with $r$ colours there exists $g\in {C\choose B}$ such that ${B\choose A}\circ g=\{f\circ g: f\in {B\choose A}\}$ is {\em{$c$-monochromatic}}, that is, $c$ restricted to  ${B\choose A}\circ g$ is constant. 

Typically a projective Fra\"{i}ss\'e class is not a Ramsey class, however, it can become one when expanded by more relations such as a linear order.
This is the case also for the class $\F$, while $\F$ is not a Ramsey class, the natural expansion $\F_<$,
as we show using Theorem \ref{Tstepup_d} and
Corollary \ref{dhe}, is a Ramsey class.

\begin{thm}\label{ordfans}
The class $\F_<$ is a Ramsey class.
\end{thm}

Let $\G$ be a projective  Fra\"{i}ss\'e family with the projective  Fra\"{i}ss\'e limit $\g$. 
Let $G={\rm Aut}(\g)$ be the automorphism group of $\g$.
We say that $\G$ is {\em rigid} if for every $A\in \G$, ${\rm Aut}(A)$ is trivial.
In Section \ref{KPTcorrespondence}, we discuss and dualize the Kechris-Pestov-Todor\v{c}evi\'c correspondence to the projective setting showing the following.
\begin{thm}\label{kpt}
The following are equivalent
\begin{enumerate}
\item The group $G$ is extremely amenable.
\item The family $\G$ is a Ramsey class and it consists of rigid elements.
\end{enumerate}
\end{thm}

For a topological group $G,$ a $G${\em-flow} (or a {\em flow} if there is no ambiguity)  is a continuous action of $G$ on a compact Hausdorff space $X,$ i.e. a continuous map $\pi:G\times X\to X$ such that $\pi(e,x)=x$ for every $x\in X$ and  
$e$ 
the identity in $G,$  and $\pi(gh,x)=\pi(g,\pi(h,x))$ for every $x\in X$ and $g,h\in G.$ When the action is understood, we write $gx$ instead of $\pi(g,x)$. We call $G$ {\em extremely amenable, if every $G$-flow has a fixed point.}
A $G$-flow is called {\em minimal} if it has no non-trivial closed $G$-invariant subsets. A  continuous map $\psi:X\to Y$ between two $G$-flows is a {\em homomorphism} if $\psi(gx)=g(\psi(x))$ for every $g\in G$ and $x\in X.$ The  {\em universal minimal flow} of $G$ is the unique minimal $G$-flow that has all other minimal $G$-flows as its homomorphic images. The universal minimal flow exists for every topological group and it is unique up to isomorphism.
It is easy to see that $G$ is extremely amenable if and only if the universal minimal flow of $G$ is a singleton.

Let $\lel$ and
$\lel_<$ denote, respectively, the projective Fra\"{i}ss\'e limits of $\F$ and $\F_<$, and let $\auto$ be the automorphism group of $\lel_<$.
Then $\auto$ is a closed subgroup of $\aut$, the automorphism group of $\lel$;
 see  Section \ref{preli} for more details.
Since $\F_<$ is a rigid Ramsey class, 
 Theorem \ref{kpt} 
provides the following. 

\begin{thm}\label{eaauto}
The group $\auto$ is extremely amenable.
\end{thm}

Let $H(L)$ denote the homeomorphism group of the Lelek fan $L$ with the compact-open topology. 
The group $\aut$ continuously embeds as a dense subgroup into $H(L)$;  see  Section \ref{preli}.
Let $H$ be the closure of $\auto$ via this embedding. 
Then Theorem \ref{eaauto} will imply Theorem \ref{ea}. 

\begin{thm}\label{ea}
The group $H$ is extremely amenable.
\end{thm}
In Proposition \ref{hl}, we  identify $H$  with the group $H(L_<)$ of homeomorphisms that preserve the order coming from the one on $\lel_<$.


\subsection{A generalisation of Gowers' Ramsey theorem to multiple operations}\label{resultsRamsey}

To prove theorems stated in Section \ref{dynamo}, we will need to generalize the finite version of  Gowers' Ramsey theorem (Theorem \ref{finG}).
For this we will prove Theorem \ref{Tstepup_d} and Corollary \ref{dhe}. In this section, we will state Theorem \ref{Tstepup_d}. First we will introduce the
 necessary notation.
 
Let $\mathbb{N}=\{1,2,3,\ldots\}$ denote the set of natural numbers 
{(we will follow the convention that 0 is not a natural number)} and, for the remainder of this section, fix
 $k\in\mathbb{N}$. 
 For a function 
$p:\mathbb{N}\to \{0,1,\ldots,k\}$, we define the {\em support }  $\supp(p)$ of $p$ to be the set
$\{l\in\mathbb{N}: p(l)\neq 0\}$. 
Let  \[
\fin_k=\{p:\mathbb{N}\to \{0,1,\ldots,k\}: |\supp(p)|<\infty {\rm \ and\  } (\exists l\in\mathbb{N}) \left(p(l)=k\right)\}, \] 
 and, for each $n\in\mathbb{N}$, let
\[\fin_k(n)=\{p:\mathbb{N}\to \{0,1,\ldots,k\}: \supp(p)\subset \{1,2,\ldots,n\}\}.\]

	\noindent  We equip $\fin_k$  and each $\fin_k(n)$ with a partial semigroup operation $+$ defined 
	for $p$ and $q$ whenever $\max(\supp(p))<\min( \supp(q))$ by  $(p+q)(x)= p(x)+ q(x)$. 


Gowers' Theorem (Theorem \ref{GG}, below)  involves 
a {\em tetris} operation $T: \fin_k\to \fin_{k-1}$ defined by
\[  T(p)(l)=\max\{0,p(l)-1\}.\]
We define, for every $0<i\leq k$, 
an operation $T^{(k)}_i:\fin_k\to\fin_{k-1}$ 
that behaves like the identity up to the value $i-1$ and like tetris above it  as follows.
\[
		T^{(k)}_i(p)(l) =
\begin{cases}
   p(l) &\text{if } p(l)<i \\
   p(l)-1  &\text{if } p(l)\geq i.
\end{cases}
 \]
 We also define $T^{(k)}_0=\id\restriction_{\fin_k}$. 
 It may seem more natural to denote the identity
by $T^{(k)}_{k+1}$ or $T^{(k)}_\infty$, only for notational convenience later on we will be using
 $T^{(k)}_0$.
Note that in our notation, 
$T^{(k)}_1$ is the usual Gowers tetris operation.
 When the context is clear we will usually drop superscripts and write $T_i$ rather than $T^{(k)}_i$.




A sequence  
$B=(b_s)_{s\in\mathbb{N}}$ is called a {\em block sequence} if for every 
$i\in \mathbb{N}$ 
\[
{\rm max}(\supp(b_i)) < {\rm min}(\supp(b_{i+1})).
\]
Analogously, we define a finite block sequence $B=(b_s)_{s=1}^m$ and we call $m$ the {\em length} of the sequence.
We let $\fin_k^{[d]}$ denote the set of all block sequences in $\fin_k$ of length $d$ and similarly
we define $\fin_k^{[d]}(n)$.

Let $B$ 
be a block sequence in $\fin_k$ (finite or infinite).   
Let $P_k$ denote the 
product $\prod_{j=1}^k \{0,1,\ldots,j\}.$
For any $I$ such that $(0,\ldots,0)\in I\subset P_k$ and for $\vec{i}=(i(1),\ldots,i(k))\in I,$ denote 
\[
T_{\vec{i}}=T_{i(1)}\circ\ldots \circ T_{i(k)}.
\]
Let $\left< B\right>_{I}$ denotes the partial subsemigroup of $\fin_{k}$ consisting of elements of the form
\[
\sum_{s=1}^l T_{\vec{i}_s}(b_{s}),
\]
where $l$ is a natural number, $\vec{i}_s\in I,$ $b_s\in B$, for $s=1,\ldots,l,$ and there is some $s$ such that 
all the entries of $\vec{i}_s$ are 0.

For a set $X$ and $r\in\mathbb{N}$, we will often call a function $c: X\to \{1,2,\ldots,r\}$ a {\em colouring}. We say that $A\subset X$ is {\em $c$-monochromatic},
or just {\em monochromatic}, if $c\restriction A$ is constant. 

 Let us  state Gowers' Ramsey theorem in this language.
\begin{thm}[Gowers, \cite{G}]\label{GG}
Let $c:\fin_k\to \{1,2,\ldots,r\}$ be a colouring. Then there exists an infinite  block sequence $B$ in $\fin_k$ such that $\left< B \right>_{\prod_{i=1}^k\{0,1\}}$ is $c$-monochromatic. 
\end{thm}

The finite version of Gowers' theorem (Theorem \ref{finG}) 
can be deduced by a simple compactness argument. 

\begin{thm}\label{finG}
Let $k,m,r$ be natural numbers. Then there exists $n$ such that for every colouring 
$c:\fin_k(n)\to \{1,2,\ldots,r\}$  there is a  block sequence $B$ of length $m$ in $\fin_k(n)$ such that $\left< B \right>_{\prod_{i=1}^k\{0,1\}}$ is $c$-monochromatic.
\end{thm}

As a consequence of our main Ramsey result, Theorem \ref{Tstepup_d},
we will obtain the following generalisation of the finite Gowers theorem to all $T_i$'s.

\begin{cor}\label{Gowers}
Let $k,m,r$ be natural numbers. Then there exists a natural number $n$ such that for every colouring $c:\fin_k(n)\to \{1,2,\ldots,r\}$  there is a block sequence $B$ of length $m$ in $\fin_k(n)$ such that $\left< B \right>_{P_k}$ is $c$-monochromatic.
\end{cor}

In order to state Theorem \ref{Tstepup_d} in full generality, we need 
a few more pieces of notation. From the proof of  Theorem \ref{ordfans} it will be clear why this is the theorem we need.
For $l> k,$ let  $P_{k+1}^l=\prod_{j=k+1}^{l} \{1,2,\ldots,j\},$ 
and let $P_{k+1}^k$ 
 contain only the constant sequence $(0,\ldots,0)$.
Note that if $p\in\fin_l$ and $\vec{i}\in P^l_{k+1}$, then 
$T_{\vec{i}}(p)\in\fin_k$.

Let $l\geq k$ and let $B=(b_s)_{s=1}^m$ be a block sequence in $\fin_l$.
Let $T_{\vec{i}}(B)$ denote the block sequence
$(T_{\vec{i_s}}(b_s))_{s=1}^m$, for $\vec{i_s}\in P^l_{k+1}$, $s=1,\ldots,m$. 
Let $\left<\bigcup_{\vec{i}\in P_{k+1}^{l}} T_{\vec{i}}(B)\right>_{P_{k}}$ denote the partial subsemigroup of 
$\fin_{k}$ consisting of elements of the form
\[
\sum_{s=1}^m    T_{\vec{t}_s}    \circ T_{\vec{i}_s}(b_{s}),
\]
where $\vec{i}_1,\ldots,\vec{i}_m\in P_{k+1}^{l}$,  $\vec{t}_1,\ldots,\vec{t}_m\in P_k$,
and there is an $s$ such that all entries of $\vec{t}_s$ are 0.  
Let $\left<\bigcup_{\vec{i}\in P_{k+1}^{l}} T_{\vec{i}}(B)\right>_{P_{k}}^{[d]}$ be the set of all block sequences in  
$\left<\bigcup_{\vec{i}\in P_{k+1}^{l}} T_{\vec{i}}(B)\right>_{P_{k}}$ of length $d$.


\begin{thm}\label{Tstepup_d}
Let $k\geq 1$. Then for every $d,$ every $m\geq d,$ every $ l\geq k,$ and every $ r$,  there exists a natural number $n$ such that for every colouring $c:\fin^{[d]}_k(n)\to \{1,2,\ldots,r\},$ there is a block  sequence $B$ in $\fin_{l}(n)$ of length $m$ such that the partial semigroup 
$\left< \bigcup_{\vec{i}\in P_{k+1}^l}T_{\vec{i}} (B)\right>_{P_k}^{[d]}$ is $c$-monochromatic. Denote the smallest such $n$ by $G_d(k,l,m,r).$
\end{thm}

Notice that setting $k=l$  and $d=1$ in Theorem \ref{Tstepup_d}, 
and observing that $\left<\bigcup_{\vec{i}\in P_{k+1}^{k}} T_{\vec{i}}(B)\right>_{P_{k}}=\left< B \right>_{P_k}$, 
we obtain 
Corollary \ref{Gowers}.

 We prove 
Theorem \ref{Tstepup_d} in Section  \ref{mainRamsey} and use it to derive Theorem \ref{ea} in Section~\ref{dynamics}.

To motivate here the statement of Theorem \ref{Tstepup_d}, let us see how to an epimorphism between structures in $\F_<$
we can associate an element in $\fin_k^{[d]}$.
To each $f\in {C\choose A}$, we associate  $f^*=(p^f_i )_{i=1}^d\in \fin_k^{[d]}(n)$ such that 
\[
		\supp(p^f_i)=\{j:a^1_i\in f(c_j)\}
	\]
		and for $j\in \supp(p_i^f)$
\[
		p_i^f(j)=z 
		\ \iff \ f(c_j^N)=a_i^z,
\]
where $a_1,\ldots, a_d$ and $c_1,\ldots, c_n$ are the increasing enumerations of branches in $A$ and $C$, respectively.
As $f\to f^*$ is not injective, to prove Theorem \ref{ordfans}, we will need not only Theorem \ref{Tstepup_d}, but also another 
Ramsey theoretic statement -- Corollary \ref{dhe}. 

In the proof of Theorem \ref{Tstepup_d}, we generalize methods introduced by  Tyros \cite{T}, who recently gave a direct constructive proof of the finite version of  Gowers' Ramsey theorem, providing upper bounds on $n$.
Independently of Tyros, a proof of the finite version of  Gowers' theorem was   presented by Ojeda-Aristizabal \cite{OA}.
On the other hand, the only known proof of the infinite Gowers Ramsey theorem \cite{G} uses the Galvin-Glazer method of idempotents in a compact right-topological semigroup of ultrafilters.
During the time this paper was under revision,  Lupini \cite{Lu}  proved the infinite version of Corollary \ref{Gowers}.

\section{Preliminaries}\label{preli}

\subsection{A construction of the Lelek fan}\label{Lelekcons} 

For completeness, we include the construction of the Lelek fan from  \cite{BK}, and we refer the reader to that article for any details we omit here.

Given a first-order language $\mathcal{L}$ that consists of relation symbols $r_i$  with arity $m_i$,  $i\in I$, and function symbols $f_j$, 
  with arity $n_j$, $j\in J$,
a \emph{topological $\mathcal{L}$-structure} is a compact zero-dimensional second-countable space $A$ equipped with
closed relations $r_i^A\subset A^{m_i}$ and continuous functions $f_j^A: A^{n_j}\to A$, $i\in I, j\in J$.
A~continuous surjection $\phi: B\to A$ between two topological $\mathcal{L}$-structures is an
 \emph{epimorphism} if it preserves the structure, that is, for a function symbol $f$ in $\mathcal{L}$ of arity $n$ and $x_1,\ldots,x_n\in B$ we require:
\[
 f^A(\phi(x_1),\ldots,\phi(x_n))=\phi(f^B(x_1,\ldots,x_n));
\]
and for a relation symbol $r$ in $\mathcal{L}$ of arity $m$ and $x_1,\ldots,x_m\in A$  we require:
\begin{equation*}
\begin{split}
&  r^A(x_1,\ldots,x_m) \\ 
&\iff \exists y_1,\ldots,y_m\in B\left(\phi(y_1)=x_1,\ldots,\phi(y_m)=x_m, \mbox{ and } r^B(y_1,\ldots,y_m)\right).
\end{split}
\end{equation*}
The if and only if condition in preservation of relations by epimorphism allows us to obtain connected spaces as natural quotients
of inverse limits of (finite) topological structures. 

By an \emph{isomorphism}  we mean a bijective epimorphism.

Let $\mathcal{G}$ be a countable family of finite topological $\mathcal{L}$-structures. We say that $\mathcal{G}$ is a \emph{ projective Fra\"{i}ss\'{e} family}
if the following two conditions hold:

(JPP) (the joint projection property) for any $A,B\in\mathcal{G}$ there are $C\in \mathcal{G}$ and epimorphisms from $C$ onto $A$ and from $C$ onto $B$;

(AP) (the amalgamation property) for $A,B_1,B_2\in\mathcal{G}$ and any epimorphisms $\phi_1: B_1\to A$ and $\phi_2: B_2\to A$, there exists $C\in\mathcal{G}$ with epimorphisms
 $\phi_3: C\to B_1$ and $\phi_4: C\to B_2$ such that $\phi_1\circ \phi_3=\phi_2\circ \phi_4$.

A topological $\mathcal{L}$-structure $\mathbb{G}$ is a \emph{ projective Fra\"{i}ss\'{e} limit } of  a projective Fra\"{i}ss\'{e} family $\mathcal{G}$ if the following three conditions hold:

(L1) (the projective universality) for any $A\in\mathcal{G}$ there is an epimorphism from $\mathbb{G}$ onto $A$;

(L2) for any finite discrete topological space $X$ and any continuous function
 $f: \mathbb{G} \to X$ there are $A\in\mathcal{G}$, an epimorphism $\phi: \mathbb{G}\to A$, and a function
$f_0: A\to X$ such that $f = f_0\circ \phi$;

(L3) (the projective ultrahomogeneity) for any $A\in \mathcal{G}$ and any epimorphisms $\phi_1: \mathbb{G}\to A$ 
and $\phi_2: \mathbb{G}\to A$
there exists an isomorphism $\psi: \mathbb{G}\to \mathbb{G}$ such that $\phi_2=\phi_1\circ \psi$.

\begin{rem}\label{coveri}
{\rm It follows from (L2) above that if  $\mathbb{G}$ is the projective Fra\"{i}ss\'{e} limit of $\mathcal{G}$, then every finite open cover can be {\em refined by an epimorphism, i.e. for every open cover
 $\mathcal{U}$}
of $\mathbb{G}$  
there is an epimorphism  $\phi:\mathbb{G}\to A$,  for some  
  $A\in\mathcal{G}$, such that for every $a\in A$, $\phi^{-1}(a)$ is contained in an open set in $\mathcal{U}$.  }
\end{rem}

\begin{thm}[Irwin-Solecki, \cite{IS}]\label{is}
 Let $\mathcal{G}$ be a projective Fra\"{i}ss\'{e} family of finite topological $\mathcal{L}$-structures. Then:
\begin{enumerate}
 \item there exists a projective Fra\"{i}ss\'{e} limit of $\mathcal{G}$;
\item any two projective Fra\"{i}ss\'{e} limits of $\mathcal{G}$ are isomorphic.
\end{enumerate}
\end{thm}

Let $\mathcal{G}$ be a projective Fra\"{i}ss\'{e} family of topological $\mathcal{L}$-structures and let $\mathbb{G}$ be a topological $\mathcal{L}$-structure.
 We say that $\mathbb{G}$
has the {\em extension property} (with respect to $\mathcal{G}$) if
for every $A,B\in\mathcal{G}$ and epimorphisms $\phi_1: B\to A$ and $\phi_2: \mathbb{G}\to A$, there is  an epimorphism 
$\psi: \mathbb{G}\to B$ such that $\phi_2=\phi_1\circ \psi$.

Similarly as for  the (injective) Fra\"{i}ss\'{e} theory, one can show the following.
\begin{prop}\label{12ep}
Let $\mathcal{G}$ be a projective Fra\"{i}ss\'{e} family.  
 If a topological $\mathcal{L}$-structure $\mathbb{G}$ satisfies 
 properties (L1) and  (L2),  and  it has the extension property with respect to $\mathcal{G}$, then  $\mathbb{G}$  is the projective Fra\"{i}ss\'{e}  limit of $\mathcal{G}$.
\end{prop}

Below we describe the projective Fra\"{i}ss\'{e} family $\F$ that we used to construct the Lelek fan in \cite{BK}.

Recall that by a 
 {\em fan} we mean an  undirected connected simple graph with all loops, with no cycles of the length greater than one,
with a distinguished point $r$, called the {\em root}, such that all elements other than $r$ have degree at most 2. On a  fan $T,$ there is a natural partial tree order
$\preceq_T$: for $t,s\in T$ we let $s\preceq_T t$ if and only if $s$ belongs to the path connecting $t$ and the root.
We say that $t$ is a {\em successor} of $s$ if $s\preceq_T t$ and $s\neq t$.
It is an {\em immediate successor} if additionally there is no $p\in T$, $p\neq s,t$  with $s\preceq_T p\preceq_T t$.

A {\em chain} in a fan $T$ is a subset of $T$ on which the order $\preceq_T$ is linear.
A {\em branch } of a fan $T$ is a maximal chain in  $(T,\preceq_T)$. 
If $b$ is a branch in $T$, we will sometimes write $b=(b^0,\ldots,b^n)$, where $b^0$ is the root of $T$, and
$b^i$ is an immediate successor of $b^{i-1}$, for every $i=1, 2, \ldots, n$. In that case, $n $ will be called the {\em height} of the branch $b$.

 Let $\mathcal{L}=\{R\}$ be the language with $R$ a binary relation symbol. 
For $s,t\in T$ we let $R^T(s,t)$ if and only if $s=t$ or
 $t$ is an immediate successor of $s$.
 Let $\F$ be the family of all finite  fans with all branches of the same height, viewed as topological $\mathcal{L}$-structures, equipped with the discrete topology.
 Every fan in $\F$ is specified by the height of its branches and its width, that is, the number of its branches.

\begin{rem}
{\rm For two fans $(S,R^S)$ and $(T,R^T)$ in $\F$, a function $\phi: (S,R^S)\to (T,R^T)$ is an epimorphism if and only if it is a surjective homomorphism, i.e. for every $s_1,s_2\in S$, $R^S(s_1,s_2)$ implies $R^T(\phi(s_1),\phi(s_2))$.
 }
\end{rem}

We list a few relevant results obtained in \cite{BK}.

\begin{prop}
The family $\F$ is a projective Fra\"{i}ss\'{e} family.
 \end{prop}

By Theorem \ref{is}, there exists a unique Fra\"{i}ss\'{e} limit of $\F$, which we denote by $\lel=(\lel, R^{\lel})$. Let $ R_S^{\lel}$ be the symmetrization of $ R^{\lel}$, that is, 
$ R_S^{\lel}(s,t)$ if and only if $ R^{\lel}(s,t)$ or $  R^{\lel}(t,s)$, for  $s,t\in\lel$.

\begin{thm}
The relation $ R_S^{\lel}$ is an equivalence relation which has only one and two element equivalence classes.
\end{thm}

\begin{thm}
The quotient space $\lel/R^{\lel}_S$ is homeomorphic to the Lelek fan $L.$
\end{thm}

We denote by $\aut$  the group of all automorphisms of $\lel$, that is, the group of all homeomorphisms of $\lel$ that preserve the relation $R^{\lel}$.
This is a topological  group when equipped with the compact-open topology inherited from $H(\lel)$,
 the group of all homeomorphisms of the Cantor set underlying the structure $\lel$.
Since $R^\lel$ is closed in $\lel\times \lel$, the  group $\aut$ is closed in $H(\lel)$.

Note that every $h\in \aut$ induces a homeomorphism  $h^*\in H(L)$
satisfying $h^*\circ\pi(x)=\pi\circ h(x)$ for $x\in\lel$. We will frequently identify $\aut$ with the corresponding subgroup $\{h^*: h\in \aut\}$ of $H(L)$. 
Observe  that the compact-open topology  on $\aut$ is finer than the topology on $\aut$ that is inherited from the compact-open 
topology on $H(L)$.

\subsection{Ultrafilters}
In this section, we introduce the notion of an ultrafilter, which we will use in Section \ref{KPTcorrespondence}. 

\begin{definition}\label{filter}
Let $X$ be a set and let $\mathcal{E}$ be a family of subsets of $X.$ We say that $\mathcal{E}$ is a {\em filter} on $X$ if 
\begin{itemize}
\item[(1)] whenever $A\in\mathcal{E}$ and $B\supset A,$ then also $B\in\mathcal{E}$ and 
\item[(2)] for every $A,B\in\mathcal{E}$ also $A\cap B\in \mathcal{E}.$
\end{itemize}
The family $\mathcal{E}$ is an {\em ultrafilter} if in addition
\begin{itemize}
\item[(3)]
 for every $A\subset X$ either $A\in\mathcal{E}$ or $X\setminus A\in\mathcal{E} $ (but not both).
\end{itemize}
An ultrafilter is {\em free} if it does not contain a singleton.
\end{definition}
\begin{rem}\label{ultrapartition}
{\rm Note that for every ultrafilter  $\mathcal{E}$ on $X$, $A\in \mathcal{E}$ and a
 partition of $A$  into $A_1,\ldots,A_n$,  there is exactly one $i=1,\ldots, n$
such that $A_i\in\mathcal{E}$. }
\end{rem}

\begin{rem}\label{remultra}
{\rm Any family satisfying the condition (2) in Definition \ref{filter} 
can be extended to a filter by simply adding 
all supersets, and every filter can be extended to an ultrafilter by Zorn's lemma. }
\end{rem}

\section{Dualization of the   Kechris-Pestov-Todor\v{c}evi\'c  correspondence}\label{KPTcorrespondence}
In this section, we prove Theorem \ref{kpt} that   dualizes the   Kechris-Pestov-Todor\v{c}evi\'c   correspondence between extreme amenability 
of automorphism groups of countable ultrahomogeneous linearly ordered structures
and the structural Ramsey theory (Theorem 4.5 in \cite{KPT}), which was further extended by Nguyen van Th\'e (Theorem~1 in \cite{NVT}) to structures that need not be linearly ordered.

Let $\G$ be a projective  Fra\"{i}ss\'e family with the projective  Fra\"{i}ss\'e limit $\g$. 
Let $G={\rm Aut}(\g)$ be the automorphism group of $\g$ equipped with the compact-open topology.

We first prove an analogue of Proposition 3 in \cite{NVT}.

\begin{prop}\label{infRamsey}
Suppose that $\G$ is rigid.
Then the following are equivalent.
\begin{itemize}
\item[(1)] The class $\G$ is a  Ramsey class.
\item[(2)] For every $A,B\in \G$ and every colouring $c:{\mathbb{G}\choose A}\to \{1,2,\ldots,r\}$ there exists $\psi\in {\mathbb{G} \choose B}$ such that ${B\choose A}\circ\psi$ is monochromatic.
\end{itemize}
\end{prop}

\begin{proof}
Since $\g$ is projectively universal, (1) easily implies (2).

For the reverse implication, suppose that (2) holds, but there are $A,B\in \G$ for which the Ramsey property fails, i.e., for every $C\in \G$ there exists a colouring $\chi_C:{C\choose A}\to \{1,2,\ldots,r\}$ such that for no $\gamma\in {C\choose B}$ the set ${B\choose A}\circ \gamma$ is monochromatic.

We first show that there is    a free ultrafilter $\u$ on 
$\bigcup_{D\in\G} {\g\choose D}$
such that for every   $D\in \G$ and every 
$\phi\in {\g\choose D}$ 
we have
\[
K_{\phi}=\bigcup_{C\in \G}\{\psi:\psi\in {\g\choose C} \  \exists \psi'\in {C\choose D} \ \rm{such \ that \ } \phi=\psi'\circ\psi\}\in \u.
\]

 Suppose that we have 
$\phi_0:\g\to D_0$ and $\phi_1:\g\to D_1.$ 
 By (L2), we can find an $E\in\G$ and an epimorphism $\psi:\g\to E$ such that $\{\psi^{-1}(e):e\in E\}$ refines both $\{\phi_i^{-1}(d):d\in D_i\}$ for $i=1,2.$ Then clearly $K_{\psi}\subset K_{\phi_0}\cap K_{\phi_1},$ so by Remark \ref{remultra}  such a $\mathcal{U}$ exists.

 Now, for  $\phi\in {\g\choose A},$ we will write $K_{\phi}$ as a disjoint union $K_{\phi}^1\cup K_{\phi}^2\cup\ldots\cup K_{\phi}^r$
 where
\[
K^{\varepsilon}_{\phi}=\bigcup_{C\in \G}\{\psi\in {\g\choose C} : \  \exists \psi'\in {C\choose A} \rm{such \ that \ }  (\psi'\circ \psi=\phi) \ \& \ (\chi_C(\psi')=\varepsilon) \},
\]
for $\varepsilon=1,2,\ldots,r$.
Note that since $\psi$ is surjective, then $\psi'$ if it exists, it is unique.
We define a colouring $c:{\g\choose A}\to \{1,2,\ldots,r\}$ by $c(\phi)=\varepsilon$ if and only if $K^{\varepsilon}_{\phi}\in\u.$ Note that $c$ is well defined by Remark \ref{ultrapartition}. We claim that for no $\delta\in {\g\choose B},$ the collection ${B\choose A}\circ\delta$ is $c$-monochromatic. Suppose 
on the contrary that there is a  $\delta$  such that ${B\choose A}\circ \delta$ is $c$-monochromatic in a colour $\varepsilon_0.$ Then the set
\[
\bigcap_{\alpha\in {B\choose A}} K_{\alpha\circ \delta}^{\varepsilon_0} \cap K_{\delta}
\]
belongs to $\u;$ 
in particular it is nonempty containing an element $\xi\in {\g\choose C}$ for some $C\in\G.$ 
Since $\xi\in K_{\delta}$,  that there is $\xi'\in {C\choose B}$ such that $\delta=\xi'\circ\xi.$
Therefore for every  $\alpha\in {B\choose A}$, we have  $\alpha\circ\delta=(\alpha\circ\xi')\circ\xi.$
Since $\xi\in K_{\alpha\circ\delta}^{\varepsilon_0},$ we obtain that $\chi_C(\alpha\circ\xi')=\varepsilon_0$ for every $\alpha\in{B\choose A}$. 
This implies that the set ${B\choose A}\circ \xi'$ is $\chi_C$-monochromatic in the colour $\varepsilon_0$-- a contradiction.
\end{proof}

To prove Theorem \ref{kpt}, we follow the approach to extreme amenability via syndetic sets from the dissertation \cite{B} of the first author.

For $A\in \G$ and an epimorphism $\phi:\mathbb{G}\to A,$ let 
\[
G_{\phi}=\{g\in G: \forall a\in A \ g(\phi^{-1}(a))=\phi^{-1}(a)\}
\]
be the pointwise stabilizer of $\phi.$ Equivalently, $G_{\phi}=\{h\in G: \phi\circ h=\phi\}.$ It is easy to see that $G_{\phi}$ is an open, and therefore clopen subgroup of $G.$ The collection 
\[\{G_{\phi}: \phi\in {\mathbb{G}\choose A}, A\in\G\}  \]
forms a basis at the identity of $G$.

Note that for every $g\in G,$ $G_{\phi}g=\{h:\forall a\in A\  h^{-1}(\phi^{-1}(a))=g^{-1}(\phi^{-1}(a))\}$, that is,  $h\in G_{\phi}g$ if and only if $\phi\circ h=\phi\circ g.$ Projective ultrahomogeneity of $\g$ provides a natural bijective identification of $G/G_{\phi}$ with ${\mathbb{G}\choose A}$ via $G_{\phi}g\mapsto  \phi\circ g.$ 
Similarly, $gG_{\phi}=\{h\in G: \forall a\in A \ h(\phi^{-1}(a))=g(\phi^{-1}(a))\},$ that is, $h\in gG_{\phi}$ if and only if $\phi\circ h^{-1}=\phi\circ g^{-1}.$

We also introduce the setwise stabilizer $G_{(\phi)}$ of $\phi,$ that is, the clopen subgroup
\[
G_{(\phi)}=\{h: h(\{\phi^{-1}(a):a\in A\})= \{\phi^{-1}(a):a\in A\}\}.
\]
By the projective ultrahomogeneity of $\g$,
$h\in G_{(\phi)}$ if and only if for some automorphism $\psi$ of $A$  we have $\phi\circ h=\phi\circ \psi.$

\begin{definition}
A subset $A$ of a group $G$ is called \emph{syndetic} if there exist finitely many $g_1,\ldots,g_n\in G$ such that $\bigcup_{i=1}^n g_iA=G.$
\end{definition}

The following lemma characterizes extreme amenability in terms of syndetic sets.

\begin{lemma}[Barto\v{s}ov\'a \cite{B}, Lemma 11]
A topological group $G$ is extremely amenable if and only if for every pair $A,B$ of syndetic subsets of $G$ and every open neighbourhood $V$ of the identity in $G$ we have $VA\cap VB\neq \emptyset.$
\end{lemma}

In the lemma above, it is sufficient to only consider open sets $V$ taken from a neighbourhood basis of the identity in $G$.
Since for  an epimorphism $\phi:\mathbb{G}\to A$  we have  $G_{\phi}(G_{\phi} A)=G_{\phi} A,$ we immediately obtain the following equivalence.

\begin{lemma}\label{eaautolemma}
Let $G$ be a topological group that admits a neighbourhood basis at the identity consisting of open subgroups. 
 Then the following are equivalent.
\begin{itemize}
\item[(1)] $G$ is extremely amenable.
\item[(2)] For every clopen subgroup $H$ of $G$ and every $K\subset G$, at most one of 
$HK$ and $G\setminus HK$ is syndetic. 
\end{itemize}
\end{lemma}

\begin{proof}[Proof of Theorem \ref{kpt}]
(1) $\Rightarrow$ (2)
We first prove that $\G$ is rigid. 
Let $A\in \G$ and pick an epimorphism $\phi:\mathbb{G}\to A$ by projective universality. Then $G_{(\phi)}/G_{\phi}$ is a finite discrete space of cardinality $\rm{Aut}(A)$ with a natural transitive continuous action of $G_{(\phi)}$ given by $g(G_{\phi}h)=G_{\phi}hg^{-1}.$ Being an open subgroup of $G,$  $G_{(\phi)}$  is extremely amenable by Lemma 13 in \cite{BPT}, and therefore $|\rm{Aut}(A)|=|G_{(\phi)}/G_{\phi}|=1.$


Secondly, we show
that $\G$ is a Ramsey class.
Let $A\in\G$ and let $c:{\mathbb{G}\choose A}\to \{1,2,\ldots,r \}$ be a colouring. 
We view $c$ as a point in the compact space $X=\{1,2,\ldots,r\}^{{\mathbb{G}\choose A}}$ of all colourings of ${\mathbb{G}\choose A}$ by $r$ colours. We consider $X$ with the natural action of $G$ given by $g\cdot d(\phi)=d(\phi \circ g^{-1}).$ Let $Y$ be the closure of the orbit of $c.$ 
Since $G$ is extremely amenable, the induced action of $G$ on $Y$ has a fixed point $e$. By the projective ultrahomogeneity of $\mathbb{G}$, $G$ acts transitively on ${\mathbb{G}\choose A},$ and consequently $e$ must be constant, say with  the range $\{i\}\subset\{1,2\ldots,r\}.$  
Let $B\in \G$ and pick a $\gamma\in {\mathbb{G}\choose B},$ which exists by the projective universality of $\mathbb{G}$. Since $e\in \overline{Gc},$ there is $g\in G$ such that $c\restriction{{B\choose A}\circ\gamma\circ g}=e\restriction{{B\choose A}\circ\gamma}$, and therefore $c$ on ${B\choose A}\circ(\gamma\circ g)$ is  constant. Since $\G$ is rigid,  Proposition \ref{infRamsey} concludes the proof.

(2) $\Rightarrow$ (1)
Striving for a contradiction, suppose that $G$ is not extremely amenable. In the light of Lemma \ref{eaautolemma}, it means that there are  $A\in \G,$ an epimorphism $\phi:\mathbb{G}\to A$  
and $K_0,K_1\subset G$ such that both
$G_{\phi}K_0$ and $G\setminus G_{\phi}K_0=G_{\phi}K_1$ are syndetic. 
 Let $g_1,\ldots, g_n\in G$ witness  syndeticity of both, i.e., 
\[
\bigcup_{i=1}^n g_iG_{\phi}K_0=G=\bigcup_{i=1}^n g_i G_{\phi}K_1.
\]
Let $\phi_i:\mathbb{G}\to A$ be given by  $\phi_i=\phi \circ g_i^{-1}.$ Since $\G$ is rigid, we can apply the property (L2) to a disjoint clopen refinement of the cover $\{\phi_i^{-1}(a):a\in A,i=1,\ldots,n\}$ of $\mathbb{G}$ and find $B\in \G,$ an epimorphism $\gamma:\mathbb{G}\to B$ and surjections $\gamma_i:B\to A$ such that $\phi_i=\gamma_i\circ \gamma.$ Since $\gamma$ and $\phi_i$'s are epimorphisms, so are $\gamma_i$'s.

Define a colouring $c:{\mathbb{G} \choose A}\to \{0,1\}$ by $c(\psi)=\varepsilon$ if and only if whenever $k$ satisfies $\phi\circ k=\psi$ we have $k\in G_{\phi}K_{\varepsilon}.$ Let us remark that $c$ is well-defined as $\phi\circ k=\phi\circ l$ if and only if $G_{\phi}k=G_{\phi}l.$
By the Ramsey property, there is an epimorphism $\gamma':\mathbb{G}\to B$ such that ${B\choose A}\circ \gamma'$ is monochromatic in a colour $\varepsilon_0\in\{0,1\},$ {in particular, $c(\gamma_i\circ \gamma')=\varepsilon_0$ for every $i.$ Since $\mathbb{G}$ is projectively ultrahomogeneous, there is $g\in G$ such that $\gamma'=\gamma\circ g.$  We have that
\[
\gamma_i\circ\gamma'= \gamma_i \circ\gamma\circ g = \phi_i \circ g = \phi \circ g_i^{-1}g,
\]
which implies $g_i^{-1}g\in G_{\phi}K_{\varepsilon_0}$ and consequently $g\in g_iG_{\phi}K_{\varepsilon_0}$ for every $i$. It means that $g\notin \bigcup_{i=1}^n g_iG_{\phi}K_{1-\varepsilon_0}$, which is a contradiction.
}
\end{proof}

\section{$\F_<$ is a Ramsey class}\label{Ramsey class}

The goal of this section is to prove Theorem \ref{ordfans}.  
We will use Theorem \ref{Tstepup_d}, which will be proved in the next section.

Let $\G$ be a projective Fra\"{i}ss\'{e}  family.
Recall that ${B \choose A}$ is the set of all epimorphisms from $B$ onto $A,$ and
denote by $\G^{[2]}$ the set of all ordered pairs $(A,B)$ of elements in $\G$ such that ${B\choose A}\neq \emptyset.$
An ordered pair $(A,B)\in \G^{[2]}$ is a {\em Ramsey pair} if  there exists $C\in \G$ such that for every colouring $c:{C \choose A} \to\{1,2,\ldots,r\}$ there is $g\in {C \choose B}$ such that
\[
{B \choose A}\circ g = \left\{h\circ g:h\in {B \choose A}\right\}
\]
is $c$-monochromatic. 
Note that the class $\G$ is a Ramsey class if every pair $(A,B)\in\G^{[2]}$ is a Ramsey pair.

As in Section \ref{Lelekcons}, let $\mathcal{L}$ be the language consisting of one binary relation symbol $R$ and let $\F$
denote the family of finite  fans considered as $\mathcal{L}$-structures. Given $A\in \F,$ we always keep in mind the underlying natural partial order $\preceq_A$.

 Take the language $\mathcal{L}_<=\{R,S\}$ expanding $\mathcal{L}$ with one binary relation symbol   $S$ 
 and let $\F_<$ be the family of all 
 $A_<=(A,R^A,S^A)$ such that $(A,R^A)\in\F$ and
for some ordering $a_1<a_2<\ldots<a_n$ of  branches in $A$ we have 
$S^A(x,y)$ if and only if there are $i\leq j$ such that $x\in a_i$ and $y\in a_j.$ 
Note that the root $r$ of $A$ belongs to every branch  so $S^A(r,x)$  for every $x\in A$ and whenever $x,y$ belong the same branch we have $S^A(x,y)$. Observe that $S^A$ extends the natural partial order on $A$.

We will frequently use the following lemma. Its proof  is straightforward.
\begin{lemma}
A function
$f: B_<\to A_<$  is an epimorphism if and only if $f: B\to A$ is an epimorphism and there exist 
$1=k_1<\ldots<k_{m+1}=n+1$ such that for every $i=1,\ldots,n$ and $s=1,\ldots,m$, if $k_s\leq i<k_{s+1}$ then $f(b_i)\subset a_s$.

\end{lemma}

Let us start with the following special case of Theorem \ref{ordfans}.
\begin{lemma}\label{Ramseypair}
Let $A,B\in \F_<$ both consist of a single branch. If  $(A,B)\in \F^{[2]}$, then $(A,B)$ is a Ramsey pair.
\end{lemma}

We will see in a moment that Lemma  \ref{Ramseypair} is   essentially a reformulation of the classical Ramsey theorem. 
We let  $N^{[j]}$ denote the collection of all $j$-element subsets of $\{1,\ldots, N\}$.
We will often write $N$ instead of $N^{[1]}$.
\begin{thm}[Ramsey]\label{RT}
Let $k,l,r$ be natural numbers. Then there exists a natural number $N$ such that for every colouring $c:N^{[k]}\to \{1,2,\ldots,r\}$ there exists a subset $X$ of $N$ of size $l$ such that $X^{[k]}$ is $c$-monochromatic. Denote by $R(k,l,r)$ the minimal such $N$.
\end{thm}


\begin{proof}[Proof of Lemma \ref{Ramseypair}]
Suppose that $A$ has $k+1$ vertices $r_A = a^0\prec_A \ldots \prec_A a^k,$ $B$ has $l+1$ vertices $r_B=b^0\prec_B\ldots \prec_B b^l$, and  $r$ is given. Let $N=R(k,l,r)$ 
and let $C\in \F_{<}$ consist of a single branch with $N+1$  vertices $r_C=c^0\prec_C \ldots \prec_C c^N$. 
Every epimorphism $f\in{C\choose A}$ can be identified with a $k$-element subset of $C$ via $\{\min \{f^{-1}(a^i)\}:i=1,\ldots,k\},$ where the $\min$ is taken with respect to $\prec_C.$
Therefore any colouring $c: {C\choose A}\to \{1,2,\ldots,r\}$ induces a colouring $d: N^{[k]}\to\{1,2,\ldots,r\}.$ Let $X=\{x^1\prec\ldots\prec x^l\}\in N^{[l]}$ 
be such that $d$ restricted to $X^{[k]}$ is constant. Define an epimorphism $\phi:C\to B$ by $\phi(c^i)=b^j$ if  $x^j\preceq i\prec x^{j+1}$ and $\phi(c^i)=b^0$ if $i\prec x^1.$
Identifying ${B\choose A}$ with $k$-element subsets of $B$ in the same manner as above, 
we can deduce that ${B\choose A}\circ \phi$ corresponds to all $l$-element subsets of 
$X$ and therefore is monochromatic.
\end{proof}

For a natural number $N$ let   $N^{[\leq j]}$ denote the collection of all at most $j$-element subsets of $\{1,\ldots,N\}.$ 
Note that $N^{[\leq j]}=\bigcup_{i=0}^j N^{[j]}$.
Let $m, r$ be natural numbers and 
 $k_1,\ldots,k_m$  be non-negative integers
and let 
\[
c: \prod_{i=1}^m N^{[\leq k_i]}\to \{1,2,\ldots,r\}
\]
be a colouring.
Given $B_i\subset N,$ $i=1,2,\ldots,m,$
we say that $c$ is {\em size-determined} on $(B_i)_{i=1}^m$ if
whenever $A_i, A_i'\subset B_i$ with $0\leq |A_i|=|A_i'|\leq k_i$ for $i=1,2,\ldots,m$, then
\[
c(A_1,\ldots,A_m)=c(A_1',\ldots, A_m').
\]

For $f\in \prod_{i=1}^m N^{[\leq k_i]}$, define $\supp(f)=\{i: f(i)\neq\emptyset\}$.
Given a natural number $d\leq m,$ let
$\left(\prod_{i=1}^m  N^{[\leq k_i]}\right)^{[d]}$ be the set of all sequences $(f_s)_{s=1}^d$
with $f_s\in \prod_{i=1}^m N^{[\leq k_i]}$ and 
$\max(\supp(f_s))<\min(\supp(f_{s+1}))$, for $s<d$. 
Then, more generally, if
\[
\chi: \left(\prod_{i=1}^m  N^{[\leq k_i]}\right)^{[d]}\to \{1,2,\ldots,r\}
\]
is a colouring and $B_i\subset N$ for $i=1,2,\ldots,m,$
we say that $\chi$ is {\em size-determined} on $(B_i)_{i=1}^m$ if whenever $(f_s)_{s=1}^d$ and $(g_s)_{s=1}^d$
are such that 
$\supp(f_s)=\supp(g_s),$ 
$|f_s(i)|=|g_s(i)|$ and
{$f_s(i), g_s(i)\subset B_i$} for every $s\leq d$ and $i\leq m$, then
\[\chi \left( (f_s)_{s=1}^d\right)=\chi \left( (g_s)_{s=1}^d\right).\]

At the end of this section, we will prove the following theorem, whose corollary  essentially  reduces Theorem \ref{ordfans} to Theorem \ref{Tstepup_d}.  
\begin{thm}\label{Tdifheights}
Let $m, r$ be natural numbers and let
$k_1,\ldots,k_m,  l_1,\ldots,l_m$ be non-negative integers such that $k_i\leq l_i$ for every $i=1,2,\ldots,m$. 
Then there exists $N$ such that for every colouring
\[
c: \prod_{i=1}^m N^{[\leq k_i]}\to \{1,2,\ldots,r\}
\]
there are $B_1,\ldots, B_m\subset N$   with $|B_i|=l_i$ such that $c$ 
is size-determined on $(B_i)_{i=1}^m$.
Denote by $S(m,k_1,\ldots,k_m,l_1,\ldots,l_m,r)$ the minimal such $N.$
\end{thm}

We are almost ready to  prove Theorem \ref{ordfans} that the class $\F_<$ is a Ramsey class. We will use
 Corollary \ref{dhe} to reduce the proof to an application of Theorem \ref{Tstepup_d}. 
Corollary \ref{dhe}  is a multidimensional version of Theorem \ref{Tdifheights}.
\begin{cor}\label{dhe}
Let $d, m, r$ be natural numbers and let
$k_1,\ldots,k_m,  l_1,\ldots,l_m$ be non-negative integers
such that $k_i\leq l_i$ for every $i=1,2,\ldots,m$.
 Then there exists $N$ such that for every colouring
\[
\chi: \left(\prod_{i=1}^m N^{[\leq k_i]}\right)^{[d]}\to \{1,2,\ldots,r\}
\]
there are $B_1,\ldots, B_m\subset N$   with $|B_i|=l_i$ such that $\chi$ 
is size-determined on $(B_i)_{i=1}^m$.
Denote by $S_d(k_1,\ldots,k_m,l_1,\ldots,l_m,m,r)$ the minimal such $N.$
\end{cor}
\begin{proof}
Let $\Gamma=\{\gamma=(\gamma(1),\ldots,\gamma(d+1))\in\mathbb{N}^{d+1}:\gamma(1)=1<\ldots<\gamma(d+1)=m\}$. To $\gamma\in \Gamma$ and $(A_1,\ldots, A_m)\in \prod_{i=1}^m  N^{[\leq k_i]},$ we associate 
 \[
 \gamma_{(A_{1},\ldots, A_{m})}=((A_1, \ldots, A_{\gamma(2)-1}),\ldots, (A_{\gamma(d)},\ldots, A_m)) \in \prod_{i=1}^m\left( N^{[\leq k_i]}\right)^{[d]},\] 
where for $1\leq i < j\leq m$ by $(A_i,\ldots, A_j)$ we mean the function supported on $[i,j]$ with the respective values $A_i,\ldots, A_j$.

Given $\chi: (\prod_{i=1}^m N^{[\leq k_i]})^{[d]}\to \{1,2,\ldots,r\},$ we define 
 $c: \prod_{i=1}^m N^{[\leq k_i]}\to \{1,2,\ldots,r\}^{\Gamma}$ by 
\[ c(A_1,\ldots, A_m)=(\chi(\gamma_{(A_1,\ldots, A_m)}))_{\gamma\in\Gamma}.
\]
Applying Theorem \ref{Tdifheights}, we get 
 $B_1,\ldots, B_m\subset N$   with $|B_i|=l_i$ such that $c$ is size-determined on $(B_i)_{i=1}^m$.
Since whenever $\gamma_{(A_1,\ldots,A_m)}=\gamma'_{(A_1,\ldots,A_m)}$ we have $c(A_1,\ldots,A_m)(\gamma)=c(A_1,\ldots,A_m)(\gamma'),$
it follows that also $\chi$ is size-determined on $(B_i)_{i=1}^m$.
\end{proof}

\begin{proof}[Proof of Theorem \ref{ordfans}]
Let $S\in \F_<$ be of height $k$ and width $d$, and let 
 $T\in \F_<$ be of height $l\geq k$  and width $m\geq d$ (so that ${T\choose S}\neq \emptyset$). Let $n=G_d(k,l,m,r)$ be as in 
Theorem \ref{Tstepup_d} 
and let $N=S_d(n,k,\ldots,k,l,\ldots,l,r)$ be as in Corollary \ref{dhe}. 
Let $U\in\F_<$ consists of $n$ branches of height $N.$ We will show that  $U$ works for $S,T$ and $r$ colours.

Let $a_1,\ldots, a_d$ and $c_1,\ldots,c_n$ be the  increasing enumerations of branches in $S$ and $U$ respectively.
Let $(a_j^i)_{i=0}^k$ be the increasing enumeration of the branch $a_j$, $j=1,\ldots,d$, 
  and let  
$(c_j^i)_{i=0}^N$  be the increasing enumeration of the branch $c_j$ for $j=1,\ldots,n.$

To each $f\in {U\choose S}$, we associate  $f^*=(p^f_i )_{i=1}^d\in \fin_k^{[d]}(n)$ such that 
\[
		\supp(p^f_i)=\{j:a^1_i\in f(c_j)\}
	\]
		and for $j\in \supp(p_i^f)$
\[
		p_i^f(j)=z 
		\ \iff \ f(c_j^N)=a_i^z.
\]

We moreover associate to $f$ a block sequence of functions $(F_i^f)_{i=1}^d\in (\prod_{j=1}^n  (c_j\setminus \{c_j^0\})^{[\leq k]})^{[d]}$ to fully code $f$ 
as follows. For $j\in\supp(p_i^f),$ we let
 \[ F_i^f(j)=\{\min\{c_j^y\in c_j :f(c_j^y)=a_i^x\} : 0\prec x\preceq  p^f_i (j)\},\]
where the $\min$  is taken with respect to the partial order on the fan $U$. By definition $p_i^f(j)=|F_i^f(j)|$
and since $f$ is onto, for each $i$ there is a $j$ such that $p_i^f(j)=k$. Therefore
if $f_1^*=f_2^*$, then $|F_i^{f_1}(j)|=|F_i^{f_2}(j)|$ for all $i,j.$




Similarly, to any $g\in{U\choose T}$, we associate $g^*\in\fin_l^{[m]}(n)$ and
$(F_i^g)_{i=1}^m\in (\prod_{j=1}^n  (c_j\setminus \{c_j^0\})^{[\leq l]})^{[m]}$.

Let $c: {U\choose S}\to \{1,\ldots, r\}$ be given. 
Let $c_0$ be a colouring of $(\prod_{j=0}^n  (c_j\setminus \{c_j^0\})^{[\leq k]})^{[d]}$ 
induced by 
$c$ via
 the injection $f\mapsto (F_i^f)_{i=1}^d$,  colouring elements 
not of the form $(F_i^f)_{i=1}^d$ in an arbitrary way. 
We first apply Corollary \ref{dhe} to find $C_j\subset c_j\setminus \{c_j^0\}$ of size $l$ for $j=1,\ldots,n$ such that $c_0$ 
is size-determined on $(C_j)_{j=1}^n$.  
It follows that the colouring $c^*:\fin_k^{[d]}(n)\to \{1,2,\ldots,r\}$ given by $c^*(f^*)=c(f)$ for $f\in {U\choose S}$ which satisfy
 $(F_i^f)_{i=1}^d\in (\prod_{j=1}^n C_j^{[\leq k]})^{[d]}$ 
 is well-defined. 
Second, we apply Theorem \ref{Tstepup_d} to obtain a block sequence $D=(d_j)_{j=1}^m$ in $\fin_l^{[m]}(n)$  such that 
 $\left<\bigcup_{\vec{i}\in P_{k+1}^l} T_{\vec{i}}(D)\right>^{[d]}_{P_k}$ is $c^*$-monochromatic.
 
 Let $g\in{U\choose T}$ be any epimorphism such that  $g^*=D$ and
 $(F_i^g)_{i=1}^m\in (\prod_{j=1}^n  C_j^{[\leq l]})^{[m]}.$ 
Then for every $h\in{T\choose S}$, we have  $(h\circ g)^*\in \left<\bigcup_{\vec{i}\in P_{k+1}^l} T_{\vec{i}}(D)\right>^{[d]}_{P_k}$ and $(F_i^{h\circ g})_{i=1}^d\in(\prod_{j=1}^n C_j^{[\leq k]})^{[d]}$.
Since  $\left<\bigcup_{\vec{i}\in P_{k+1}^l} T_{\vec{i}}(D)\right>^{[d]}_{P_k}$ is $c^*$-monochromatic,
we can conclude that ${T\choose S}\circ g$ is $c$-monochromatic.

\end{proof}

\begin{proof}[Proof of Theorem \ref{Tdifheights}]
Let $m, r$ be natural numbers and 
$k_1,\ldots,k_m,  l_1,\ldots,l_m$ be fixed non-negative integers such that $k_i\leq l_i$ for every $i=1,2,\ldots,m$. 
We proceed by a double  induction on $m'$ and on $k'\leq k_{m'}$, where at each step we apply Theorem \ref{RT}.
We prove the following statement:
Given $1\leq m'\leq m$ and  $0\leq k'\leq k_{m'}$, there exists $N$ such that for every colouring
\[
c: \prod_{i=1}^{m'-1} N^{[\leq k_i]}\times N^{[\leq k']}\to \{1,\ldots,r\}
\]
there are $B_1,\ldots, B_{m'}\subset N$   with $|B_i|=l_i$ such that $c$ 
is size-determined on $(B_i)_{i=1}^{m'}$. 

When $m'=1 $ and $k'=0$, there is nothing to prove.
Let $m'=1$  and assume that 
the statement of the theorem holds for $0\leq k'<k_1$, we will prove it for $k'+1$.
Let $N'=S(1,k',l_1,r)$ and $N=R(k'+1, N',r).$ 
Let
\[
c: N^{[\leq (k'+1)]} \to \{1,2,\ldots, r\}
\]
be a given colouring.
Consider the restricted colouring $d=c\restriction_{N^{[k'+1]}}$ and apply Theorem \ref{RT} to find $B\subset N$ of size $N'$ such that $B^{[k'+1]}$ is $d$-monochromatic. 
By the inductive hypothesis applied to $c\restriction{  B^{[\leq k']}}$, we obtain the desired $B_1$.

Suppose that the statement of the theorem is true for $m'-1$ and we shall prove it for $m'.$ 
When $k'=0$, simply take $N=S(m'-1, k_1,\ldots,k_{m'-1},l_1,\ldots,l_{m'-1},r)$. 
Assume that the result is true for $k'<k_{m'}$, and we will prove it for $k'+1$.
Set \[N'=S(m', k_1,\ldots,k_{m'-1}, k', l_1,\ldots,l_{m'-1}, l_{m'}, r)\] and
\[N''=S(m'-1,k_1,\ldots,k_{m'-1}, N',\ldots, N',r).\]
Denote by $\alpha$ the set of all colourings of $\prod_{i=1}^{m'-1} N''^{[\leq k_i]}$ with colours $1,\ldots,r$.
Let 
\[N=\max\{N'',R(k'+1,N',|\alpha|)\}.\] We will show that $N$ works. 

Let
\[
c:\left(\prod_{i=1}^{m'-1} N^{[\leq k_i]}\right)\times \left( N^{[\leq (k'+1)]}\right)\to \{1,2,\ldots, r\}
\]
be an arbitrary colouring.
For every $A\subset N$ of size $k'+1$, let $c_A$ be the colouring of $\prod_{i=1}^{m'-1} N^{''[\leq k_i]}
$ induced by $c$  and $A$ in the last coordinate, i.e.
\[
c_A(A_1,\ldots,A_{m'-1})=c(A_1,\ldots,A_{m'-1},A).
\]
Define a colouring
\[
d:N_{m'}^{[k'+1]}\to\alpha
\]
by $d(A)=c_A.$

By Theorem \ref{RT}, there is a  subset $B_{m'}'\subset N$ of size $N'$ such that $B_{m'}^{'[k'+1]}$ is $d$-monochromatic
 in a colour 
$c_0:\left(\prod_{i=1}^{m'-1} N^{''[\leq k_i]}\right)\to \{1,2,\ldots, r\}$.
Applying the induction hypothesis for $m'-1,k_1,\ldots,k_{m'-1},N',\ldots,N', r$, we obtain $B_i'\subset N_i''$ of size $N'$ for $i=1,2,\ldots,m'-1$ 
such that $c_0$ is size-determined on $(B'_i)_{i=1}^{m'-1}$.

Finally, we define
\[
b=c\restriction_{\left(\prod_{i=1}^{m'-1} B_i'^{[\leq k_i]}\right)\times\left(B_{m'}'^{[\leq k']}\right)}.
\]
Using the induction hypothesis for $m'-1,k_1,\ldots,k_{m'-1},k',l_1,\ldots,l_{m'},r$, we obtain $B_i\subset B_i'$ for $i=1,2,\ldots,m'$ such that $|B_i|=l_i$ and $b$ and therefore $c$ are size-determined on $(B_i)_{i=1}^{m'}$.

\end{proof}

\section{Gowers' Ramsey theorem for multiple operations}\label{mainRamsey}

In this section, we provide a proof by induction of our main Ramsey result, Theorem \ref{Tstepup_d}. In order to perform the induction, we generalize Tyros' notions of a type  and of a pyramid in $\fin_k(n)$ to sequences in $\fin_k^{[d]}(n)$.

Let $A=(a_i)_{i=1}^m$ be a block sequence in $\fin_1.$ We can identify each $a_i$ with 
the characteristic function $\chi(a_i)$ of its support. We define 
\[
\fin_k(A)=
\left\{\sum_{i=1}^m j_i\cdot\chi(a_i): j_i\in \{0,1,\ldots, k\} \ \& \ \exists i  \ (j_i=k) \right\}. 
\]
Let $\fin_k^{[d]}(A)$ denote the set of all block sequences in $\fin_k(A)$ of length $d$.

A function $\phi:\{1,\ldots, m\}\to \{1,\ldots,k\}$ is a {\em type  of length $m$ over $k$} if $\phi(i)\neq\phi(i+1)$ for  $i=1,2,\ldots,m-1,$
and for  some $i$, $\phi(i)=k$. Note that $\phi\in \fin_k(m)$.
Let $n$ be a natural number.
For every type $\phi$ of length $m\leq n$ over $k$ and every block sequence $B=(b_i)_{i=1}^{m}$ in $\fin_1(n),$ 
\[
{\rm{map}}(\phi, B)=\sum_{i=1}^m \phi(i)\cdot\chi(b_i)
\]
belongs to $\fin_{k}(n)$.

On the other hand, for every $p\in \fin_{k}(n),$ there exist a unique natural number $m\leq n,$ a type $\phi$ of length  $m$ over $k$, and a block sequence $B=(b_i)_{i=1}^{m}$ in $\fin_1(n)$ of length $m$ such that
\[
p={\rm{map}}(\phi,B).
\]
We call this $\phi$ the {\em type} of $p$ and denote it by $\tp(p)$.
We say that $p,q\in \fin_k$ are of the same type if $\tp(p)=\tp(q).$

Tyros \cite{T} used the following lemma about types to obtain his constructive proof of the finite version of Gowers' theorem.

\begin{lemma}[Tyros \cite{T}]\label{KT}
For every triple $k,m,r$ of  natural numbers, there exists $n$ such that
for every colouring $c:\fin_k(n)\to \{1,2,\ldots,r\},$  there is a block sequence $A$ of length $m$ in $\fin_1(n)$ such that any two elements in $\fin_k(A)$ of the same type have the same colour. 
\end{lemma}

We extend the notion of a type to sequences in $\fin_k^{[d]}.$ 
We say that $\phi=(\phi_1,\ldots,\phi_d)$ is a {\em type of length $m$ over $k$} if 
each $\phi_i$ is a type of length $m_i$ over $k$ such that 
 $\sum_{i=1}^d m_i=m$. 
For $\bar{p}=(p_1,\ldots,p_d)\in \fin_k^{[d]},$ we let the {\em type} of $\bar{p}$ be $(\tp(p_1),\ldots,\tp(p_d)).$

 Note that  
 given $d\leq n$ and  $A\in\fin_{1}^{[n]},$
there is a natural bijection between $\fin_k^{[d]}(A)$ and the set of pairs $(B,\phi),$ where $B\in \fin_1^{[m]}(A)$ and 
$\phi\in\prod_{i=1}^d \fin_k(m_i)$, for some $m_1,\ldots, m_d$ such that $\sum_{i=1}^d m_i=m$,  is 
  a type of length $m$ over $k$ for some  $m\leq n.$  As in the case of dimension~1, the $p\in \fin_{k}^{[d]}(A)$ that corresponds to $(B,\phi)$ will be denoted by ${\rm{map}}(\phi, B)$.

We will prove 
a multidimensional version of Lemma \ref{KT} using a finite
version of  the Milliken-Taylor theorem \cite{M,Ta}.

\begin{thm}[Milliken-Taylor]\label{mt}
Given natural numbers $m\geq d$ and $ r,$ there exists $n$ with the following property: For every finite block sequence
$A\in\fin_1$ of length at least $n$ and every colouring of $\fin_1^{[d]}(A)$ by $r$ colours
there exists $B\in \fin_1^{[m]}(A)$ such that $\fin_1^{[d]}(B)$ is monochromatic.
We denote the smallest such $n$ by $MT_d(m,r).$
\end{thm}

\begin{lemma}\label{lemmaTyros}
Let $k$ and $d\leq m,$ and $r$ be natural numbers. Then there exists $n$ such that for every colouring
 $c: \fin_k^{[d]}(n)\to\{1,2,\ldots,r\}$, there is a block sequence $A$ in $\fin_1(n)$ of length $m$ such that any two elements in 
 $\fin_k^{[d]}(A)$  of the same type have the same colour. We denote the smallest such $n$ by $T_d(k,m,r).$
\end{lemma}
 
\begin{proof}
Let $\mathcal{T}$ be the set of all types of sequences in $\fin_k^{[d]}$ of length at most $m$ and let $\alpha$ be the cardinality of the set $X$ of all colourings of  
$\mathcal{T}$ 
by $r$ colours. 
Let $n=MT_m(2m-d,\alpha)$   be as in Theorem \ref{mt} and let $c: \fin_k^{[d]}(n)\to\{1,2,\ldots,r\}$ be a colouring.
Let $q: \fin_1^{[m]}(n)\to\{1,\ldots, r\}^\mathcal{T}$ be  the colouring given by
  \[q(B)(\phi)=c({\rm{map}}(\phi,(b_i)_{i=1}^{l_\phi})),\]
  where $l_\phi$ denotes  length of the type $\phi$, $B=(b_i)_{i=1}^m\in\fin_1^{[m]}(n),$
  and $\{1,\ldots, r\}^\mathcal{T}$  denotes the set of functions from $\mathcal{T}$ to $\{1,\ldots, r\}$. 
By Theorem \ref{mt}, we can find a block sequence $A'$ 
of length $2m-d$
such that $\fin_1^{[m]}(A')$ is
 $q$-monochromatic and let $A$ be the initial segment of $A'$ of length $m.$ We will show that $A$ is as desired. Indeed, 
let $\bar{p}_1,\bar{p}_2\in \fin_k^{[d]}(A)$ be of the same type $\phi,$
and let $A_1, A_2$ be the block sequences in $\fin_1(A)$ for which $\bar{p}_1={\rm{map}}(\phi,A_1)$ and $\bar{p}_2={\rm{map}}(\phi,A_2)$.
Since $\phi$ has length between $d$ and  $m$
 and since $A$ is an initial segment of $A'\in\fin_k^{[2m-d]}(n),$ 
 we can choose  $A'_1, A'_2\in \fin_1^{[m]}(A')$
  such that $A_1$ is an initial segment of $A'_1$ and $A_2$ is an initial segment of $A'_2$.
  It follows that 
  \[c(\bar{p}_1)=q(A'_1)(\phi)=q(A'_2)(\phi)=c(\bar{p}_2).\]

\end{proof}

 Another piece needed for the induction in the proof of Theorem \ref{Tstepup_d} is a pair of
two lemmas capturing how $T_1$ commutes with $T_i$'s.
The proof of the  first lemma 
is an immediate calculation.

\begin{lemma}\label{L1}
\begin{enumerate}
\item If $1\leq j< l$ and $p\in\fin_l$, we have $T_j\circ T_1(p)=T_1\circ T_{j+1}(p)$ and $T_0\circ T_1=T_1\circ T_0$.
\item For $\vec{t}\in P_k$, $p\in\fin_l$, $T_{\vec{t}}\circ T_1(p)=T_1\circ T_{\vec{t}+1}(p),$  where $\vec{t}+1\in P_{k+1}$ is such that
\begin{align*}
(\vec{t}+1)(1) &=   0 \\
(\vec{t}+1)(x+1) &=     \vec{t}(x)+1 \text{ when } 1\leq x\leq k. 
\end{align*}  
\item If $\vec{t}=\vec{j} ^\frown \vec{i}$ with $\vec{j}\in P_{k-1}$ and $\vec{i}\in P_k^{l-1}$, then 
 $\vec{t}+1=\vec{j'} ^\frown \vec{i'}$ where $\vec{j'}\in P_{k}$ and $\vec{i'}\in P_{k+1}^{l}$.
\item If $\vec{j}\in P_k$ and $\vec{i}\in P^l_{k+1}$, then
\[ T_1\circ T_{\vec{j}}\circ T_{\vec{i}}(p)=T_{\vec{j'}}\circ T_{\vec{i'}}\circ T_1(p)\]
for some $\vec{j'}\in P_{k-1}$, $\vec{i'}\in P_k^{l-1}$.
\end{enumerate}

\end{lemma}


The  second lemma will easily follow from Lemma \ref{L1}
Let us recall the definitions of $P_k$ and $P^l_{k+1}$ from Introduction and observe that for any $1\leq k$
\[ \bigcup_{\vec{i}\in P_{k+1}^{k}} T_{\vec{i}}(B)=B.\] 


\begin{lemma}\label{lemmaTstepup}
Let $B=(b_s)_{s=1}^m$ be a block sequence 
in $\fin_l(n).$  
Then for $2\leq k\leq l,$ we have
\begin{equation*}\label{e1}
T_1\left<\bigcup_{\vec{i}\in P_{k+1}^{l}} T_{\vec{i}}(B)\right>_{P_{k}}=\left<\bigcup_{\vec{i}\in P_{k}^{l-1}} T_{\vec{i}}\circ T_1 (B)\right>_{P_{k-1}}.
\end{equation*} 
In particular, if $2\leq k$   and $k=l$
 \begin{equation*}\label{e2}
T_1\left<B\right>_{P_{k}}=\left<T_1 (B)\right>_{P_{k-1}}.
\end{equation*}
\end{lemma}





An element $c\in \fin_l$ is called a {\em pyramid of height $l$} if for some block sequence $A=(a_j)_{j=-(l-1)}^{j=l-1}$ in $\fin_1^{[2l-1]}$,
we have 
\[c=\sum_{j=-(l-1)}^{l-1} (l-|j|) \cdot \chi(a_{j}).\]

Observe that if $c$ is a pyramid of height $l$, 
 $\vec{i}\in P_l$  and $j$  is the number of zero entries in $\vec{i,}$ then $T_{\vec{i}}(c)$ is a pyramid of height $j.$ 
 Note that some of the ``steps'' in the pyramid $T_{\vec{i}}(c)$ may have disappeared and others may have become longer.
  If $k<l$ and $\vec{i}\in P_{k+1}^l$, then $T_{\vec{i}}(c)$ is a pyramid of height $k$.
In particular, for every $\vec{i}$ in $P_k$ or in $P_{k+1}^l,$ we have that 
\begin{equation}\label{eq}
T_{\vec{i}}(c)(\text{min }\supp (T_{\vec{i}}(c))) = 1 = T_{\vec{i}}(c) (\text{max }\supp (T_{\vec{i}}(c))) \tag{*}.
\end{equation}

Let $C=(c_i)_{i=1}^n$ be a block sequence of $n$ pyramids of height $k.$ For $p\in \left< C\right>_{P_k},$ let 
\[
\supp_C(p)=\{i: \supp(p)\cap\supp(c_i)\neq\emptyset\}.
\]
For $i\in \supp_C(p), $ we define $\he_i(p)=\max \{p(x):x\in \supp(p)\cap\supp(c_i) \}$, while for $i\notin \supp_C(p)$,  we let $\he_i(p)=0$.
Then for $1\leq i\leq n,$  $\he(p)(i)=\he_i(p)$  defines a function in            
$\fin_k(n)$.
For $\bar{p}=(p_1,\ldots,p_d)\in \fin_k^{[d]}(C),$ we let  $\he(\bar{p})=(\he(p_1),\ldots,\he(p_d))\in\fin_k^{[d]}(n).$ 


\begin{lemma}\label{def}
Let $C_1, C_2$ be two  block sequences  of $n$ pyramids of height $k$, let  $p\in\left< C_1 \right>_{P_{k}}$
and $q\in\left< C_2 \right>_{P_{k}}.$
Then 
$\he(T_1(p))=\he(T_1(q))$ iff 
$\tp(p)=\tp(q).$
\end{lemma}
\begin{proof}
Observe that $i\in\supp_{C_1}(T_1(p))$ iff $\he_i(T_1(p))>0$ iff  $\he_i(p)=\he_i(T_1(p))+1,$ while $i\notin\supp_{C_1}(T_1(p))$ iff $\he_i(p)=0$ or $\he_i(p)=1$. 
Analogously 
for $q$.   
The statement  therefore follows by equation (\ref{eq}). 
\end{proof}


The following lemma shows that if the statement of Theorem \ref{Tstepup_d} holds then it remains true after  replacing $\fin_k(n)$ by the partial semigroup generated by a block sequence of $n$  pyramids of height $k$. It will be essential for the induction in the proof of Theorem \ref{Tstepup_d}.

\begin{lemma}\label{followup_d}
Suppose that 
\begin{enumerate}
\item[(**)]
for every $1\leq k,$ every $d,$ every $m\geq d,$ every $ l\geq k,$ and every $r$, there exists $n$ such that for every colouring $c:\fin^{[d]}_k(n)\to\{1,2,\ldots,r\}$ there is a block sequence $B$ in $\fin_l(n)$ of length $m$ such that $\left< \bigcup_{\vec{i}\in P_{k+1}^l}T_{\vec{i}} (B)\right>_{P_k}^{[d]}$ is 
$c$-monochromatic.
\end{enumerate}

Let $C=(c_i)_{i=1}^n$ be a block sequence of $n$ pyramids of height $l$. Then for every colouring $e:\left<\bigcup_{\vec{i}\in P_{k+1}^l}T_{\vec{i}} (C)\right>_{P_k}^{[d]}\to\{1,2,\ldots,r\}$    
satisfying that  
$\he(\bar{p})=\he(\bar{q})$ implies 
$e(\bar{p})=e(\bar{q})$, 
there is a block sequence $D\in\left<C\right>_{P_l}^{[m]}$ 
such that $\left<\bigcup_{\vec{i}\in P_{k+1}^l}T_{\vec{i}} (D)\right>_{P_k}^{[d]}$ is $e$-monochromatic.
\end{lemma}

\begin{proof}
Let $C$ and $e$ be as in the statement of the theorem.
For $1\leq j\leq l,$ we define a one-to-one semigroup homomorphism
\[\iota_j:\fin_j(n)\to \left<\bigcup_{\vec{i}\in P_{j+1}^l} T_{\vec{i}}(C)\right>_{P_j}
\] by

\[q\mapsto \sum_{i\in \supp(q)} T_1^{l-q(i)}(c_i),\]
where $T^{l-q(i)}_1$ denotes the $(l-q(i))$-th iterate of $T_1.$ 
We naturally extend $\iota_j$ to $\iota_j^{[d]}$ on $\fin_j^{[d]}(n)$ by $\iota_j^{[d]}(q_i)_{i=1}^{d}=(\iota_j(q_i))_{i=1}^d.$

Let $c:\fin_k^{[d]}(n)\to \{1,2,\ldots,r\}$ be the colouring $e\circ \iota_k^{[d]}.$ By the hypothesis~(**), we can find a block sequence $B$ in $\fin_l^{[m]}(n)$ such that $\left< \bigcup_{\vec{i}\in P_{k+1}^l}T_{\vec{i}} (B)\right>_{P_k}^{[d]}$ is $c$-monochromatic in a color $\alpha$. Define $D=\iota_l^{[m]}(B)\in\left<C\right>^{[m]}_{P_l}$. 

It is easy to see that $\he(\iota_j(q))=q$ and that 
$\iota_{j-1} T_{r}(q)=T_r\iota_{j}(q)$ for $1\leq r\leq j$ and $q\in\fin_j(n)$.
It implies that whenever $\bar{p}\in \left<\bigcup_{\vec{i}\in P_{k+1}^l}T_{\vec{i}} (D)\right>_{P_k}^{[d]}$, 
$\he(\bar{p})\in   \left< \bigcup_{\vec{i}\in P_{k+1}^l}T_{\vec{i}} (B)\right>_{P_k}^{[d]},$ 
and therefore $e(\iota_k^{[d]}(\he(\bar{p})))=\alpha,$  but 
$\he(\iota_k^{[d]}(\he(\bar{p})))=\he(\bar{p})$, so also $e(\bar{p})=\alpha$.

\end{proof}



For $\bar{p}=(p_1,\ldots, p_d)\in \fin_k^{[d]}$, we define $T_1(\bar{p})$ to be $(T_1(p_1),\ldots, T_1(p_d))$.
{We are now ready to prove Theorem \ref{Tstepup_d}. Some of the ideas used in the proof also appeared in
 the proof of Theorem 1 in \cite{T}.}

\begin{proof}[Proof of Theorem \ref{Tstepup_d}]
We proceed by induction on $k.$ For $k=1$ and $d,m\geq d, l\geq k, r$ arbitrary, let 
$n=MT_d(m,r).$
 Suppose that $c:\fin^{[d]}_1(n)\to\{1,2,\ldots,r\}$ is an arbitrary colouring. 
By Theorem \ref{mt}, we can find a block sequence $A=(a_s)_{s=1}^m\in\fin_1^{[m]}(n)$ 
such that $\fin_1^{[d]}(A)$ is $c$-monochromatic.
 We define $B=(b_s)_{s=1}^m$ by $b_s=l\cdot \chi(a_s),$ so that $T_{\vec{i}}(b_s)=a_s$ for every $s$ and $\vec{i}\in P_2^l$. Then 
 \[ \left<\bigcup_{\vec{i}\in P_2^l}T_{\vec{i}}(B)\right>^{[d]}_{P_1} = \left< A \right>_{P_1}^{[d]}=\fin_1^{[d]}(A)\] 
is $c$-monochromatic.

Now, we assume that  the theorem holds for $k-1$ and we shall prove it for $k.$ Let $n'= G_d(k-1,l-1,m,r)$ be given by the induction hypothesis and let
 $n=T_d(k,n'(2l-1),r)$ be as in Lemma \ref{lemmaTyros}.

Let $c:\fin^{[d]}_k(n)\to \{1,2,\ldots,r\}$ be a given colouring. By Lemma \ref{lemmaTyros}, 
we can find a sequence $A$ in $\fin_1(n)$ of length $n'(2l-1)$ such that 
any two elements in $\fin_k^{[d]}(A)$  of the same type have the same colour. 
Let  $C=(c_i)_{i=1}^{n'}$ be the block sequence of $n'$ pyramids in $\fin_l(A),$ i.e.
\[
c_i=\sum_{j=-(l-1)}^{l-1} (l-|j|) \cdot \chi(a_{q_i+j}),
\]
where $q_i=(i-1)(2l-1)+l.$

Suppose that $\bar{p},\bar{q}\in \left< \bigcup_{\vec{i}\in P_{k+1}^{l}}T_{\vec{i}}(C)\right>_{P_{k}}^{[d]}$ 
are such that $\he(T_1(\bar{p}))=\he(T_1(\bar{q}))$. 
  Then $\tp(\bar{p})=\tp(\bar{q})$ by Lemma \ref{def} and consequently $c(\bar{p})=c(\bar{q})$ by the choice of $C$. Therefore
the colouring 
\[
c':T_1\left< \bigcup_{\vec{i}\in P_{k+1}^{l}}T_{\vec{i}}(C)\right>_{P_{k}}^{[d]} \to\{1,2,\ldots,r\},
\]
given by $c'(T_1(\bar{p}))=c(\bar{p}),$ is well-defined.

By  Lemma \ref{lemmaTstepup},
\[
T_1\left< \bigcup_{\vec{i}\in P_{k+1}^{l}}T_{\vec{i}}(C)\right>_{P_{k}} = \\
\left< \bigcup_{\vec{i}\in P_{k}^{l-1}}T_{\vec{i}}\circ T_1(C)\right>_{P_{k-1}}.
\]
Therefore 
$c' $ and the sequence of pyramids $T_1(C)$  satisfy the hypothesis of Lemma \ref{followup_d}. 

Applying the induction hypothesis together with Lemma \ref{followup_d}
we can find a block sequence
 $B'=(b'_s)_{s=1}^m$ in $\left< T_1(C)\right>_{P_{l-1}}$ such that $\left< \bigcup_{\vec{i}\in P_{k}^{l-1}}T_{\vec{i}} (B')\right>_{P_{k-1}}^{[d]}$
 is $c'$-monochromatic, say in a 
colour~$\alpha.$ 
By  Lemma \ref{lemmaTstepup}, 
there is a block sequence $B$ in $\left< C \right>_{P_l}^{[m]}$ such that $T_1(B)=B'.$ If 
$\bar{b}\in \left< \bigcup_{\vec{i}\in P_{k+1}^{l}}T_{\vec{i}}(B)\right>_{P_k}^{[d]},$
then
 $T_1(\bar{b})\in \left< \bigcup_{\vec{i}\in P_{k}^{l-1}} T_{\vec{i}}(B')\right>_{P_{k-1}}^{[d]}$,
 so $c(\bar{b})=c'(T_1(\bar{b}))=\alpha$. We can conclude that $B$ is as required.

\end{proof}
As stated in Section \ref{resultsRamsey}, setting $k=l$ and $d=1$ in Theorem \ref{Tstepup_d}, we obtain Corollary \ref{Gowers}, a generalisation of the finite version of  Gowers'  Theorem from the tetris operation $T_1$ to all the operations $T_i$.
Setting $k=l$ and letting $d$ be arbitrary in Theorem \ref{Tstepup_d}, we obtain the following generalisation of the finite version of  
the 
Milliken-Taylor theorem for $\fin_k$ (see \cite{To}, Corollary 5.26). 

\begin{cor}\label{T2}
Let $k,m,r$ and $d$ be natural numbers. Then there exists a natural number $n$ such that for every colouring $c:\fin^{[d]}_k(n)\to \{1,2,\ldots,r\}$  there is a block sequence $B$ of length $m$ in $\fin_k(n)$ such that $\left< B \right>_{P_k}^{[d]}$ is $c$-monochromatic.
\end{cor}

In an earlier version of the article we posed the following question.

\begin{question}
{\rm Does Corollary \ref{Gowers} admit an infinitary version?}
\end{question}
This question was recently solved in positive by Lupini \cite{Lu}.

\section{Applications to dynamics of $H(L)$}\label{dynamics}
In this section, we describe a natural closed subgroup $H$ of $H(L)$ and show that it is extremely amenable. 


It is not difficult to see that $\F_<$  consists of rigid elements and that it has the JPP. The proposition below thus asserts that $\F_<$ is a projective Fra\"{i}ss\'{e} class.
\begin{prop}\label{aapp}
The  family $\F_<$ has the AP.
\end{prop}

One can deduce this theorem from the JPP and from the Ramsey property, cf~\cite{KPT}, page 20.
 We will include a direct proof of Theorem \ref{aapp} in Appendix \ref{aaaa}.



Having shown that $\F_<$ is a projective  Fra\"{i}ss\'{e} class, we may now consider its projective  Fra\"{i}ss\'{e} limit $\mathbb{L}_<.$  
Let  $G=\auto$ denote the automorphism group of $\lel_<$. Combining the Kechris-Pestov-Todor\v{c}evi\'c correspondence from Section \ref{KPTcorrespondence} with $\F_<$ being a rigid Ramsey class, we obtain the that the group $\aut$ is extremely amenable.
 
\begin{proof}[Proof of Theorem \ref{eaauto}]
Follows from Theorems  \ref{ordfans} and \ref{kpt}. 
\end{proof}

Observe that the family $\F_<$ is {\em reasonable} with respect to $\F$, that is, for every $A,B\in\F$, an epimorphism 
$\phi: B\to A$,  and $A_<\in\F_<$
such that $A_<\restriction\mathcal{L}=A$, there is $B_<\in\F_<$ such that $B_<\restriction\mathcal{L}=B$ and $\phi: B_<\to A_<$
is an epimorphism. 
The proof of the following lemma uses that $\F_<$ is reasonable with respect to $\F$ and
implies that $\auto$ may be identified with a subgroup of $\aut$, cf. \cite{KPT} Proposition~5.2. 
Our proof is slightly more complicated than the one in \cite{KPT} since we do not have an analogue of the hereditary property for $\F.$
\begin{lemma}\label{lemmaidentify}
We have $\mathbb{L}_<\restriction \mathcal{L}$ is a projective Fra\"{i}ss\'e limit of $\F$ and therefore isomorphic to $\mathbb{L}.$ 
\end{lemma}
\begin{proof}   
Set $\mathbb{L}_0= \mathbb{L}_<\restriction \mathcal{L}$. By Proposition \ref{12ep},  it suffices  to show that $\mathbb{L}_0$ satisfies the
properties (L1) and  (L2)
in the definition of the projective Fra\"{i}ss\'{e}  limit, and it has
the extension property  with respect to $\F$. 
Since $\mathbb{L}_<$ has the properties (L1) and (L2) with respect to $\F_<$ and for every $A\in\F$ there is $A_<\in\F_<$ with $A_<\restriction  \mathcal{L}=A$, $\mathbb{L}_0$ has the properties (L1) and (L2) with respect to $\F.$

 To show the extension property, let $A,B\in\F$ and let $\phi_1: B\to A$ and $\eta : \mathbb{L}_0\to A$ be epimorphisms.
 By property (L2) for $\mathbb{L}_<,$ find $C_<\in\F_<$, an epimorphism $\xi :\mathbb{L}_<\to C_<$ and a  map $\phi_2: C\to A$,
 such that $C=C_<\restriction  \mathcal{L}$ and $\phi_2\circ \xi=\eta$. Note that since $\xi:\mathbb{L}_0\to C$ and
  $\eta : \mathbb{L}_0\to A$ are epimorphisms,
 so is  $\phi_2: C\to A$.  From the AP for $\F,$ find $D\in\F$, $\psi_1: D\to B$ and $\psi_2: D\to C$ such that $\phi_1\circ\psi_1=\phi_2\circ\psi_2$.
 Take any $D_<\in\F_<$ with $D_<\restriction  \mathcal{L}=D$ such that $\psi_2: D_<\to C_<$ is an epimorphism.
 Using the extension property for $\F_<$, find an epimorphism $\rho: \mathbb{L}_<\to D_<$ such that $\psi_2\circ\rho=\xi$.
 A simple calculation shows that the epimorphism $(\psi_1\circ\rho): \mathbb{L}_0\to B$ is as needed, i.e. it satisfies
 $\eta=\phi_1\circ (\psi_1\circ\rho)$.
\end{proof}

In the light of Lemma \ref{lemmaidentify}, from now on we will think of $\mathbb{L}_<$ as $(\lel,S^{\lel})$ and identify $\auto$ with a closed subgroup of $\aut.$ To demonstrate that Theorem \ref{eaauto} is not trivial, it is appropriate to show that the group $\auto$ is not a singleton. It will follow from the projective ultrahomogeneity of $\lel.$
Indeed, let $A\in\F_<$ be the fan of height 1 and width 3, with branches $a_1=(a_1^0,a_1^1)$, $a_2=(a_2^0,a_2^1)$, and
$a_3=(a_3^0,a_3^1)$, where $r_A=a_1^0=a_2^0=a_3^0$ is the root. Let $B\in\F_<$ be the fan 
of height 1 and width 2 with branches $b_1=(b_1^0,b_1^1)$ and $b_2=(b_2^0,b_2^1)$,
 where $r_B=b_1^0=b_2^0$ is the root.
Let $\phi:\lel_< \to A$ be an arbitrary epimorphism. Let $\alpha_1: A\to B$ be the epimorphism
given by $r_A\mapsto r_B$, $a_1^1\mapsto b_1^1$, $a_2^1\mapsto b_1^1$, and $a_3^1\mapsto b_2^1$, and let $\alpha_2: A\to B$ be the epimorphism
given by $r_A\mapsto r_B$, $a_1^1\mapsto b_1^1$, $a_2^1\mapsto b_2^1$, and $a_3^1\mapsto b_2^1$. Then the projective ultrahomogeneity applied to 
$\alpha_1\circ\phi, \alpha_2\circ\phi : \lel_<\to B$ provides us with a non-trivial automorphism of $\lel_<$.

Let $\pi:\lel\to\lel/R^{\lel}_S\cong L$ be the natural quotient map. 
Since $\lel_<=(\lel, S^{\lel})$ is a topological $\mathcal{L}_<-$structure, the relation $S^{\lel}$ is closed and consequently 
$\leq_L=\pi(S^{\lel})$ is a closed binary relation on $L$ which is reflexive and transitive, that is, it is a preorder.
We will call the Lelek fan equipped with $\leq_L$ the {\em preordered Lelek fan} 
and denote it by $L_<$.

In Section \ref{Lelekcons}, we pointed out that $\pi$ induces an injective continuous homomorphism which we denote by
$\pi^*$ from  $\aut$ 
 onto a subgroup of $H(L)$.
\begin{definition}\label{defhl}
We define the following two subgroups of $H(L)$
\begin{eqnarray}
\indent H &=& \overline{\pi^*(\auto)}^{H(L)} \\
 H(L_<) &=& \{h\in H(L): {\rm\ for\ every\ } x, y\in L\  (x\leq_L y\implies h(x)\leq_L h(y))\}.
 \end{eqnarray}
\end{definition}

\begin{prop}\label{hl}
We have $H=H(L_<)$.
\end{prop}

In the proof of Proposition \ref{hl}, we will use  Lemma \ref{cover}, an analog of Lemma 2.14 from \cite{BK} but for $L_<$ instead of $L$. 

\begin{lemma}\label{cover}
Let $d<1$ be any metric compatible with the topology on $L$.  
Let $\varepsilon>0$ and let   $v$ be the top of $L_<$. Then there are $A_<=(A,S^A)\in\F_<$ and  
an open cover $(U_a)_{a\in A_<}$ of $L_<$ 
such that  
\begin{itemize}
\item[(C1)] for each $a\in A_<$, $\text{diam}(U_a)<\varepsilon$, 
\item[(C2)]  for every $I=[v,e]$ where $e$ is an endpoint and $v$ is the top point, $\{U_a\cap I:a\in A_<\}$ is a cover consisting of intervals such that 
each set $(U_a\cap I)\setminus  (U_{a'}\cap I)$ is connected and
whenever $U_a\cap I, U_{a'}\cap I$ have a non-empty intersection and there is $y\in U_{a'}\cap I$ with $U_a\cap I\subset [v,y]$ we have $R^{A_<}(a,a').$
\item[(C3)] for every $a\in A_<$ there is $x\in  L_<$ such that $x\in U_a\setminus(\bigcup \{U_{a'}: a'\in A_<, a'\neq a\}),$
\item[(C4)] for each $x,y\in L_<$ and $a,b\in A_<$ , if $x\leq_L y,$
$x\in U_a$ and $y\in U_b$, then  $S^{A_<}(a,b)$. 
\end{itemize}
\end{lemma}

\begin{proof}[Proof of Lemma \ref{cover}]
Let $\mathcal{U}$ be a finite open $\frac{\varepsilon}{3}$-cover of $L_<$ and let $\mathcal{V}=\{\pi^{-1}(U): U\in\mathcal{U}\}$.
Using (L2) in the definition of the projective Fra\"{i}ss\'{e} limit, find $A_<\in\F_<$ and an epimorphism
$\phi:\lel_<\to A_<$ that refines $\mathcal{V}$. The set
\[ \mathcal{C}_1=\{V_a=\pi(\phi^{-1}(a)): a\in A_<\}.\]
is a closed $\frac{\varepsilon}{3}$-cover of $L_<$ that satisfies all properties (C1)-(C4).
Since $L$ is compact,  the distance between any 
$D,E \in \mathcal{C}_1,$ $D\cap E=\emptyset$ is positive, that is, $\text{inf}\{d(x,y):x\in D, y\in E\}>0$.
So we can find $0 < \delta <\frac{\varepsilon}{3}$ such that for every $D,E\in \mathcal{C}_1$,  we have 
\[B(D,\delta)\cap B(E,\delta)\neq\emptyset  \ \ \ \iff \ \ \  D\cap E\neq\emptyset,\]
where for $X\subset L_<$ we set $B(X,\delta)=\{y\in L_<: \exists_{x\in X} \ d(y,x)<\delta\}$.
Then the cover 
$ \mathcal{C}_2=\{ U_a=B(V_a,\delta): a\in A_<\}$
 satisfies the properties (C1)-(C4) as well.
\end{proof}

\begin{proof}[Proof of Proposition \ref{hl}]
For every $h\in\auto,$ we have that $h^*\in H$. Since $H(L_<)$ is closed, it follows that $H\subset H(L_<)$.

To show the converse, take $h\in H(L_<)$ and $\varepsilon>0$. 
Let $d<1$ be any metric compatible with the topology on $L$ and let $d_{\sup}$ be the corresponding supremum metric on $H(L)$.
We will find $\gamma\in\auto$ such that $d_{\sup}(h,\gamma^*)<\varepsilon$.
Let $A_<\in\F_<$ and let  $(U_a)_{a\in A_<}$ be an open cover of $L_<$  as in Lemma \ref{cover}. 
Since $h$ is uniformly continuous, we can assume additionally that for each $a\in A_<$, $\text{diam}(h[U_a])<\varepsilon$. As $h$ is a homeomorphism, $(h[U_a])_{a\in A_<}$ also satisfies conditions (C2)-(C3) of Lemma \ref{cover}. Finally, $h\in H(L_<)$ ensures that  $(h[U_a])_{a\in A_<}$ satisfies (C4).

Consider the open covers $\{V^1_a := \pi^{-1}(U_a):a\in A_<\}$ and $\{V^2_a:=\pi^{-1}(h[U_a]):a\in A_<\}$  of $\lel_<$.
By the property (L2), we can find $B_<\in \F_<$ and epimorphisms
$\phi_i: \lel_<\to B_<$ for $i=1,2$ that refine the cover $(V^i_a)_{a\in A_<}$.
Define $\alpha_i: B_<\to A_<$ by 
$b\mapsto\max\{a\in A_<: \pi\circ\phi_i^{-1}(b)\subset U_a\},$
where the maximum is taken with respect to the natural partial order on $A.$ 
Let $\psi_i: \lel_<\to A_<$ be the composition  $\alpha_i\circ\phi_i.$
We will show that $\psi_i,$ $i=1,2$ are epimorphisms. Since $\phi_i$ are continuous, so are $\psi_i,$
 and by (C3) they are  onto. 
The property (C2) implies that if $x,y\in \lel_<$ satisfy $R^{\lel_<}(x,y)$ then 
$R^{A_<}(\psi_i(x),\psi_i(y))$, $i=1,2$.
 Finally, (C4) provides that if $S^{\lel_<}(x,y)$ then $S^{A_<}(\psi_1(x),\psi_1(y))$. Since $(h[U_a])_{a\in A_<}$ also satisfies (C2)-(C4), the same is true for $\psi_2,$ and we can conclude    
that $\psi_1,\psi_2$ are epimorphisms.

The projective ultrahomogeneity 
 gives us $\gamma\in\auto$ such that $\psi_1=\psi_2\circ \gamma .$
It remains to show that $d_{\sup}(h,\gamma^*)<\varepsilon$. 
Pick any $x\in \lel_<$ and let $a=\psi_1(x)$.  
Then
\[ \gamma^*(\pi(x))\in\gamma^*(\pi\circ\psi_1^{-1}(a))=\pi\circ\gamma\circ\psi_1^{-1}(a)=\pi\circ\psi_2^{-1}(a)\subset h[U_a].\] 
It means that $\gamma^*(\pi(x)),$ $h(\pi(x))\in h[U_a]$, and since $\text{diam}(h[U_a])<\varepsilon$, we get the required conclusion.

\end{proof}

\appendix

\section{}\label{aaaa}

We present below a proof of Theorem \ref{aapp}.
\begin{proof}[Proof of Theorem \ref{aapp}]
Take  $A,B,C\in\F_<$ together with epimorphisms $\phi_1: B\to A$ and $\phi_2: C\to A$.

For clarity, we start with the simplest case, which will be applied in the induction further on.
\begin{claim*}
Assume that  $A, B$ and $C$ all consist of one branch only. Then there are $D\in \F_<$ and epimorphisms $\psi_1:D\to B$ and $\psi_2:D\to C$ such that $\phi_1\circ\psi_1=\phi_2\circ\psi_2.$
\end{claim*}
\begin{proof}[Proof of Claim]
Suppose that $A$ has height $l$ and enumerate it as ${a^0,a^1,\ldots, a^l}$ with $a^0$ the root and $R^A(a^i,a^{i+1})$ for $i=0,1,\ldots,l-1.$ For  $0\leq i\leq l$ and $\varepsilon=1,2,$ let $I^{\varepsilon}_i=\phi^{-1}_{\varepsilon}(a^i)$ 
and let $m_i= {\rm{max}}\{|I_i^1|,|I_i^2|\}.$
Let $D\in \F_<$ have a single branch of height $M=m_0 + \ldots + m_l$  and write it as $D=\bigcup_{i=0}^l K_i$ with $|K_i|=m_i$ and all elements in $K_i$ preceding elements in $K_{i+1}$ in the natural order. Then $\psi_{\varepsilon}$ mapping $K_i$ onto $I^{\varepsilon}_i$ in an $R$-preserving manner for $\varepsilon=1,2$ finish the argument. 
\end{proof}

Second, we deal with the situation when $A=\{a^0,\ldots,a^l\}$  is a single branch 
and $B$ and $C$ have $m$ and $n$ branches, respectively.

The relation $S^B$ induces an ordering $b_1<b_2<\ldots<b_m$ of branches in $B$, and 
the relation $S^C$ induces an ordering $c_1<c_2<\ldots<c_n$ of branches in $C$. We will perform induction on $k=1,\ldots,m+n$ assuming that $\phi_1\restriction b_m$ and $\phi_2\restriction c_n$ both map onto $A$. Later we will see how to eliminate this assumption.
In  step $k,$ we will construct a branch $d_k$ of $D$  together with homomorphisms $\psi_{1,k}:d_k\to B$ and $\psi_{2,k}:d_k\to C$ such that $\phi_1\circ \psi_{1,k}=\phi_2\circ \psi_{2,k}.$ We will ensure that the image of $\bigcup\{\psi_{\varepsilon,i}:i=1,\ldots,k, \varepsilon=1,2\}$ contains the first $k_B$ branches of $B$ and the first $k_C$ branches of $C$ for some $k_B$ and $k_C$
such that $k=k_B+k_C.$ Each of the steps resembles the proof of Claim.

\noindent {\bf Step $\boldsymbol{k=1}$.}
Since $\phi_1(b_1)$ and $\phi_2(c_1)$ are initial segments of the single branch of $A,$ we may without loss of generality assume that $\phi_1(b_1)\subset\phi_2(c_1).$ Applying Claim with $b_1$ in place of $B$ and  $\phi_2^{-1}(\phi_1(b_1))$ in place of $C,$ we may find a branch $d_1$ and homomorphisms $\psi_{1,1}:d_1\to b_1$ and $\psi_{2,1}:d_1\to c_1$ with $\psi_{1,1}$ onto $b_1,$ $\psi_{2,1}$ onto $\phi_2^{-1}(\phi_1(b_1))$ such that $\phi_1\circ \psi_{1,1}=\phi_2\circ\psi_{2,1}.$

\noindent {\bf Step $\boldsymbol{k+1}$.} 
Suppose that we have constructed  branches $d_i$ with homomorphisms $\psi_{1,i}: d_i\to B$ and $\psi_{2,i}:d_i\to C$ for $i=1,\ldots,k$ as required.
Pick branches $b$ in $B$ and $c$ in $C$ with the smallest indices such that they are not contained in the image of $\bigcup\{\psi_{\varepsilon,i}:i=1,\ldots,k, \varepsilon=1,2\}.$
Since $\phi_1(b)$ and $\phi_2(c)$ are   initial segments of the only
 branch of $A$, we can assume without loss of generality that 
either ($\phi_1(b)\subset \phi_2(c)$ and $\phi_1(b)\neq \phi_2(c)$)
or ($\phi_1(b)= \phi_2(c)$ and $m-k\geq n-k$). With $b$ in place of $b_1$ and $c$ in place of $c_1,$ we may proceed in the same way as in Step $k=1$ to obtain a branch $d_{k+1}$ and homomorphisms $\psi_{1,k+1}:d_{k+1}\to B$ and $\psi_{2,k+1}:d_{k+1}\to C$ with $\psi_{1,k+1}$  onto $c$, $\psi_{2,k+1}$ onto $\phi_{2}^{-1}(\phi_1(b))$ and $\phi_1\circ\psi_{1,k+1}=\phi_2\circ\psi_{2,k+1}.$ This finishes the induction step.

Note that $\phi_1\restriction{b_m},$ $\phi_2\restriction{c_n}$ being  onto $A$ 
allows us to proceed as above for $(n+m)$-many steps, in each step covering a new branch in $B$ or $C$. We remark that we may ensure that all $d_k$'s have the same height.

Let $D\in \F_<$ be the union of the branches $(d_k)_{k=1}^{m+n}$ with their roots identified and $S^D$ induced by the order $d_1<\ldots < d_{m+n}$. Then $\psi_1:D\to B$ and $\psi_2:D\to C$ given by $\psi_{\varepsilon}\restriction{d_k}=\psi_{\varepsilon,k}$ for $\varepsilon=1,2$ are as required.


When $\phi_1\restriction{b_1}$ and $\phi_2\restriction{c_1}$ are onto $A$,
we proceed similarly as above  
(starting the induction with $b_m$ and $c_n$ and going backwards).

{In the case when $A$ has one branch and $B$ and $C$ are arbitrary, let $t_B$ and $t_C$ denote any branch in $B$ and $C,$ respectively, such that $\phi_1\restriction{t_B}$ and $\phi_2\restriction{t_C}$ are onto $A.$ We split $B$ into two fans, $B_1$ and $B_2,$ such that all branches in $B_1$ precede the branches in $B_2$ in the order corresponding to $S^B$ (which we denote by $B_1<B_2$) and such that $B_1\cap B_2=t_B.$ Similarly, we split $C$ into $C_1<C_2$ with $C_1\cap C_2=t_C.$}

For $B_i$, $C_i$, $\phi^i_{1}=\phi_1\restriction B_i$, and $\phi^i_{2}=\phi_2\restriction C_i$, {we} obtain  $D_i$, $\psi^i_1$, and $\psi^i_2$, 
 $i=1,2$, and observe that $D=D_1\cup D_2$ with  their roots identified and {$S^D$ inducing} $D_1< D_2,$ 
$\psi_1=\psi^1_1\cup\psi^2_1$, and $\psi_2=\psi^1_2\cup \psi^2_2$ are as required. 

In a general situation, when $A$ consists of branches $a_1<\ldots<a_s$, we follow the procedure above  for each $a_i$,
$B_i=\phi_1^{-1}(a_i)$, $C_i=\phi_2^{-1}(a_i)$, $\phi^i_{1}=\phi_1\restriction B_i$, and $\phi^i_{2}=\phi_2\restriction C_i$. For each $i,$ we obtain  $D_i$, 
$\psi^i_1:D_i\to B_i $, and $\psi^i_2:D_i\to C_i$, $i=1,\ldots, l$, and conclude that $D=D_1\cup\ldots\cup D_s$ with roots identified and such that
 $D_1<\ldots< D_s$, $\psi_1=\psi^1_1\cup\ldots \cup \psi^s_1$, and $\psi_2=\psi^1_2\cup \ldots\cup\psi^s_2$ finish the proof.

\end{proof}

 \thanks{{  \bf Acknowledgments. }{\rm 
We  would like to thank Miodrag Soki\'c, Gianluca Basso, and two anonymous referees} for a number of suggestions that  helped us to improve the presentation of the paper. We thank Martino Lupini for pointing out an error in Section 8 of an earlier version of the paper. The first author was supported by FAPESP (2013/14458-9) and FAPESP (2014/12405-8).
}

\end{document}